\documentclass[10pt]{amsart}%
\usepackage{graphicx}
\usepackage{amscd, color}
\usepackage{amsmath}
\usepackage{amsfonts}
\usepackage{amssymb}%
\providecommand{\U}[1]{\protect\rule{.1in}{.1in}}
\providecommand{\U}[1]{\protect\rule{.1in}{.1in}}
\providecommand{\U}[1]{\protect\rule{.1in}{.1in}} \textwidth 15.8cm
\theoremstyle{plain}

\newtheorem{theorem}{Theorem}[section]
\newtheorem{proposition}[theorem]{Proposition}

\newtheorem{remark}[theorem]{Remark}

\newtheorem{definition}[theorem]{Definition}
\numberwithin{equation}{section}

\begin{document}
\title[Multiple Cohen strongly summing operators]{Multiple Cohen strongly summing operators, ideals, coherence and compatibility}
\author{Jamilson R. Campos}
\address{Departamento de Ci\^{e}ncias Exatas \\
\indent Universidade Federal da Para\'{\i}ba - Campus IV \\
\indent R. da Mangueira, s/n, Centro \\
\indent Rio Tinto, 58.297-000, Brazil.}
\email{jamilson@dce.ufpb.br and jamilsonrc@gmail.com}

\begin{abstract}
Considering the successful theory of multiple summing multilinear operators as
a prototype, we introduce the classes of multiple Cohen strongly $p$-summing
multilinear operators and polynomials. The adequacy of these classes under the
viewpoint of the theory of multilinear and polynomial ideals and holomorphy types is discussed in detail. Some abstract results are also proved in the abstract setting of the \textit{full general Pietsch Domination Theorem} due to Pellegrino, Santos and Seoane-Sep\'{u}lveda.

\end{abstract}
\maketitle




\section{Introduction}


J. S. Cohen \cite{cohen73} introduced the class of strongly $p$-summing
linear operators motivated by the fact that the class of absolutely
$p$-summing linear operators is not closed under conjugation. Pietsch
(\cite{pietsch67}, page 338) shows that the identity operator from $l_{1}$ to
$l_{2}$ is absolutely $2$-summing but its conjugate, from $l_{2}$ to
$l_{\infty}$, is not absolutely $2$-summing. In his work Cohen shows that the
class of strongly $p$-summing operators characterizes the conjugates of
absolutely $p^{\ast}$-summing operators, with $1/p+1/p^{\ast}=1$.

In the context of the theory of operator ideals (\cite{pietsch80,PPPP}), it is a natural question whether the class of Cohen (linear)
operators forms a complete ideal and also how to generalize this class to
multi-ideals and polynomial ideals without loosing the essence of the original
ideal. For absolutely summing linear operators there are several types of
extensions (e. g. \cite{Matos-N, daniel11}). We mention \cite{studia},
\cite{CDM09}, \cite{CDM33}, \cite{CDM}, \cite{joilson11} as attempts to
establish general criteria that the ideals should possess to preserve
properties of the linear ideal.

The notion of Cohen summability for multilinear operators was investigated by \cite{achour07} (see also \cite{mezrag09}). In this work we introduce the class of multiple Cohen strongly $p$-summing
multilinear operators and polynomials and we show that these ideals are
coherent and compatible.

There are also two well known approaches used to study polynomials between Banach spaces. On the one hand, L. Nachbin \cite{nachbin69} introduced holomorphy types as classes of polynomials that are stable under differentiation. On the other hand, A. Pietsch \cite{pietsch80} introduced the notion of ideals of multilinear operators that immediately adapts to polynomials. Some classes of polynomials are both holomorphy types and ideals, such as nuclear and compact polynomials. We show that this fact is also true for the class of multiple Cohen strongly $p$-summing polynomials.

The paper is organized as follows: Sections \ref{ol} and \ref{om} contain
definitions and results concerning Cohen linear operators; in Section
\ref{omc} we define the notion of multiple Cohen strongly $p$-summing
multilinear operators and compare the size of this class with the class introduced in \cite{achour07}; Section \ref{ic} is devoted to show that these classes are ideals;
Section \ref{sc} deals with the notions of coherence and
compatibility in this framework and, finally, Section \ref{ht} is devoted to show that the class of multiple Cohen strongly $p$-summing polynomials is a global holomorphy type.

Although most authors, probably all, define the Cohen strongly $p$-summing
multilinear operators through inequalities, we prefer to define them primarily
using sequences, and then we show the equivalent characterizations via
inequalities. This approach has some advantages in the statements of many of
our results.

\vspace{0,15cm}

\section{Linear Cohen strongly summing operators}\label{ol}

Let us begin with a more detailed study of the class of Cohen strongly
$p$-summing operators, defined in \cite{cohen73}.

\vspace{0,15cm}

\begin{definition}
[Cohen, \cite{cohen73}]\label{defcohen} A sequence $(x_{i})_{i=1}^{\infty}$ in
a Banach space $E$ is Cohen strongly $p$-summing if the series $\sum
_{i=1}^{\infty}\varphi_{i}(x_{i})$ converges for all $(\varphi_{i}%
)_{i=1}^{\infty}\in l_{p^{*}}^{w}(E^{^{\prime}})$, with $1/p+1/p^{*}=1$.
\end{definition}

We denote by $l_{p}\langle E\rangle$ the space of Cohen strongly $p$-summing
sequences in $E$. It is possible, for reasons of management, to replace the
series $\sum_{i=1}^{\infty}\varphi_{i}(x_{i})$ in Definition \ref{defcohen} by
the series $\sum_{i=1}^{\infty}|\varphi_{i}(x_{i})|$, which we use in our
text. The proof of this result is performed by a straightforward calculation.

\vspace{0,15cm}

\begin{proposition}
Let $(x_{i})_{i=1}^{\infty}$ be a sequence in $E$. The series $\sum
_{i=1}^{\infty}\varphi_{i}(x_{i})$ converges for all $(\varphi_{i}%
)_{i=1}^{\infty}\in l_{p^{\ast}}^{w}(E^{^{\prime}})$ if and only if the series
$\sum_{i=1}^{\infty}|\varphi_{i}(x_{i})|$ converges for all $(\varphi
_{i})_{i=1}^{\infty}\in l_{p^{\ast}}^{w}(E^{^{\prime}})$.
\end{proposition}

Is not difficult to prove that $l_{p}\langle E\rangle$ is a normed space with
the norm
\[
||(x_{i})_{i=1}^{\infty}||_{C,p} = \sup_{||(\varphi_{i})_{i=1}^{\infty
}||_{w,p^{*}} \leq1}\sum_{i=1}^{\infty}|\varphi_{i}(x_{i})|
\]
and furthermore the duality $\left(  l_{p}(E)\right)  ^{^{\prime}}=l_{p^{*}}(E^{^{\prime}})$,
with $1/p + 1/p^{*}=1$, leads to following result

\vspace{0,15cm}

\begin{proposition}[Cohen, \cite{cohen73}]
\label{inclusoes} If $1\leq p\leq\infty$, then $l_{p}\langle E\rangle\subset
l_{p}(E)\subset l_{p}^{w}(E)$. Moreover, if $p=1$ then $l_{p}\langle
E\rangle=l_{p}(E)$ and if $p=\infty$ then $l_{p}(E)=l_{p}^{w}(E)\label{equal1}$ isometrically.
\end{proposition}

\vspace{0,15cm}
Note that if $T\in\mathcal{L}(E;F)$, then the operator
\[
\widehat{T}^{s}:l_{p}\left(  E\right)  \rightarrow l_{p}(F)\ \mathrm{defined}%
\ \mathrm{by}\ \left(  x_{i}\right)  _{i=1}^{\infty}\mapsto\left(  T\left(
x_{i}\right)  \right)  _{i=1}^{\infty},
\]
is well-defined and continuous. In our context, the interesting case occurs
when this type of correspondence induces a continuous operator from $l_{p}(E)$
to $l_{p}\langle F\rangle$, which motivates the definition of a Cohen strongly $p$-summing operator.

\vspace{0,15cm}

\begin{definition}
Let $1 < p \leq\infty$. An operator $T \in\mathcal{L}(E;F)$ is Cohen strongly
$p$-summing if $\left(  T(x_{i})\right)  _{i=1}^{\infty}\in l_{p}\langle
F\rangle$ whenever $(x_{i})_{i=1}^{\infty}\in l_{p}(E)$, that is, if the
operator
\begin{equation*}\widehat{T}\ :l_{p}\left(  E\right)  \rightarrow l_{p}\langle
F\rangle\ ; \left(  x_{i}\right)  _{i=1}^{\infty} \mapsto\left(  T\left(
x_{i}\right)  \right)  _{i=1}^{\infty}%
\end{equation*}
is well-defined.
\end{definition}

\vspace{0cm} We denote by $\mathcal{D}_{p}(E;F)$ the set formed by the Cohen
strongly $p$-summing operators. It is easy to show that $\mathcal{D}_{p}(E;F)$
is a subspace of $\mathcal{L}(E;F)$ and, by Proposition \ref{inclusoes},
$\mathcal{D}_{1}(E;F)=\mathcal{L}(E;F)$. The next result (essentially known)
gives some characterizations for the Cohen strongly $p$-summing operators.

\vspace{0,15cm}

\begin{proposition}
\label{prop3} For $T \in\mathcal{L}(E;F)$ and $\frac{1}{p}+\frac{1}{p^{*}}=1$,
the following statements are equivalent:

\vspace{0,15cm} $(i)$ $T$ is Cohen strongly $p$-summing;

$(ii)$ there is a $C>0$ such that
\[
\sum_{i=1}^{\infty}|\varphi_{i}(T(x_{i}))| \leq C\, ||(x_{i})_{i=1}^{\infty
}||_{p} ||(\varphi_{i})_{i=1}^{\infty}||_{w,p^{*}}\ ,
\]
whenever $(x_{i})_{i=1}^{\infty}\in l_{p}(E)$ and $(\varphi_{i})_{i=1}%
^{\infty}\in l_{p^{*}}^{w}(F^{^{\prime}})$;

$(iii)$ there is a $C>0$ such that
\begin{equation}
\label{def3}\sum_{i=1}^{m} |\varphi_{i}(T(x_{i}))| \leq C\, ||(x_{i}%
)_{i=1}^{m}||_{p} ||(\varphi_{i})_{i=1}^{m}||_{w,p^{*}}\ ,
\end{equation}
for all $m \in\mathbb{N},\ x_{i} \in E,\ \varphi_{i} \in F^{^{\prime}%
},\ i=1,...,m$ .
\end{proposition}

\begin{proof}
$(i)\Rightarrow(ii)$ Since $T$ is Cohen strongly $p$-summing, the operator
\begin{align*}
\widetilde{T}:\ \  &  l_{p^{\ast}}^{w}(F^{^{\prime}})\times l_{p}%
(E)\rightarrow l_{1}\\
&  \left(  (\varphi_{i})_{i=1}^{\infty},(x_{i})_{i=1}^{\infty}\right)
\mapsto\left(  \varphi_{i}(T(x_{i}))\right)  _{i=1}^{\infty}\ ,
\end{align*}
is well-defined and bilinear. A simple calculation shows that $\widetilde{T}$
has closed graph and hence is continuous. Hence,
\begin{equation*}
\sum_{i=1}^{\infty}|\varphi_{i}(T(x_{i}))|=\left\Vert \widetilde{T}%
((\varphi_{i})_{i=1}^{\infty},(x_{i})_{i=1}^{\infty})\right\Vert _{l_{1}}%
\leq\left\Vert \widetilde{T}\right\Vert \,||(x_{i})_{i=1}^{\infty}%
||_{p}||(\varphi_{i})_{i=1}^{\infty}||_{w,p^{\ast}}\ .
\end{equation*}

$(iii) \Rightarrow(ii)$ Let $(x_{i})_{i=1}^{\infty}\in l_{p}(E)$ and
$(\varphi_{i})_{i=1}^{\infty}\in l_{p^{*}}^{w}(F^{^{\prime}})$. Then,
\begin{align*}
\sum_{i=1}^{\infty}|\varphi_{i}(T(x_{i}))|  &  = \sup_{m} \left(  \sum
_{i=1}^{m} |\varphi_{i}(T(x_{i}))| \right) \\
&  \leq\sup_{m} \left(  C\, ||(x_{i})_{i=1}^{m}||_{p} ||(\varphi_{i}%
)_{i=1}^{m}||_{w,p^{*}} \right) \\
&  = C\, ||(x_{i})_{i=1}^{\infty}||_{p} ||(\varphi_{i})_{i=1}^{\infty
}||_{w,p^{*}}.
\end{align*}

$(ii) \Rightarrow(i)$ and $(ii) \Rightarrow(iii)$ are immediate.
\end{proof}

The smallest $C$ such that the inequality (\ref{def3}) is satisfied defines a
norm in $\mathcal{D}_{p}(E;F)$, denoted $d_{p}(\cdot)$. Furthermore, we have
$\left\Vert \widetilde{T}\right\Vert = d_{p}(T)$.

\vspace{0,15cm}

\section{Polynomial and multilinear Cohen strongly summing operators}\label{om}

In this section we study the class of Cohen strongly $p$-summing multilinear
operators; it is a natural extension of the linear case, based in the idea and
principles formulated by J. S. Cohen. The definition of this class is built
via sequences although the most common definition, such as in \cite{achour07}, is done by inequalities. The main results of this section are Theorem \ref{teo2} and Theorem \ref{teoabs}.

\vspace{0,15cm}

\begin{definition}
Let $1<p < \infty$ and $E_{j},F$ Banach spaces, $j=1,...,n$.
An operator $T\in\mathcal{L}(E_{1},...,E_{n};F)$ is Cohen strongly $p$-summing
if $\left(  T(x_{i}^{(1)},...,x_{i}^{(n)})\right)  _{i=1}^{\infty}\in
l_{p}\langle F\rangle$ whenever $\left(  x_{i}^{(j)}\right)  _{i=1}^{\infty
}\in l_{np}(E_{j})$, $j=1,...,n$, that is, the operator
\begin{align*}
\widehat{T}\  &  \ :l_{np}\left(  E_{1}\right)  \times\cdots\times
l_{np}\left(  E_{n}\right)  \rightarrow l_{p}\langle F\rangle \\
&  \left(  (x_{i}^{(1)})_{i=1}^{\infty},...,(x_{i}^{(n)})_{i=1}^{\infty
}\right)  \mapsto\left(  T\left(  x_{i}^{(1)},...,x_{i}^{(n)}\right)  \right)
_{i=1}^{\infty}
\end{align*}
is well-defined.
\end{definition}

The set of all $n$-linear Cohen strongly $p$-summing operators from $E_1 \times \cdots \times E_n$ to $F$ is denoted by $\mathcal{L}_{Coh,p}(E_{1},...,E_{n};F)$. Is simple to show that
$\mathcal{L}_{Coh,p}(E_{1},...,E_{n};F)$, provided with the usual operations,
is a subspace of $\mathcal{L}(E_{1},...,E_{n};F)$. The next result is folklore and we just sketch the proof.

\vspace{0,15cm}
\begin{proposition}
\label{prop4} For $T \in\mathcal{L}(E_{1},...,E_{n};F)$ and $1/p + 1/p^{*}
=1$, the following statements are equivalent:

\vspace{0,15cm} $(i)$ $T$ is Cohen strongly $p$-summing;

$(ii)$ there is a $C>0$ such that%

\begin{equation*}
\sum_{i=1}^{\infty}|\varphi_{i}(T(x_{i}^{(1)},...,x_{i}^{(n)}))|
\leq C \left(  \sum_{i=1}^{\infty}||x_{i}^{(1)}||^{np}\right)  ^{1/np} ...
\ \left(  \sum_{i=1}^{\infty}||x_{i}^{(n)}||^{np} \right)  ^{1/np}
\ ||(\varphi_{i})_{i=1}^{\infty}||_{w,p^{*}}\ ,
\end{equation*}
whenever $\left(  x_{i}^{(j)}\right)  _{i=1}^{\infty}\in l_{np}(E_{j})
\ ,\ j=1,...,n\ $ and $(\varphi_{i})_{i=1}^{\infty}\in l_{p^{*}}%
^{w}(F^{^{\prime}})$.

$(iii)$ there is a $C>0$ such that
\begin{equation}
\label{def1}\sum_{i=1}^{m} |\varphi_{i}(T(x_{i}^{(1)},...,x_{i}^{(n)}))| \leq
C \left(  \sum_{i=1}^{m} ||x_{i}^{(1)}||^{np} \right)  ^{1/np} ... \ \left(
\sum_{i=1}^{m} ||x_{i}^{(n)}||^{np} \right)  ^{1/np} \ ||(\varphi_{i}%
)_{i=1}^{m}||_{w,p^{*}}\ ,
\end{equation}
for all $m \in\mathbb{N},\ x_{i}^{(j)} \in E_{j},\ \varphi_{i} \in
F^{^{\prime}},\ i=1,...,m\ , \ j=1,...,n$ .
\end{proposition}

\begin{proof}
$(i) \Rightarrow(ii)$ This proof is analogous to $(i) \Rightarrow(ii)$ in Proposition \ref{prop3} by using the closed graph theorem for multilinear mappings.

$(ii) \Rightarrow(i)$ and $(ii) \Rightarrow(iii)$ are immediate.

$(iii) \Rightarrow(ii)$ Exactly the same argument used in $(iii) \Rightarrow(ii)$ of Proposition \ref{prop3}.
\end{proof}

The smallest $C$ such that (\ref{def1}) is satisfied, denoted by
$||T||_{Coh,p}$ , defines a norm in $\mathcal{L}_{Coh,p}(E_{1},...,E_{n};F)$.
Moreover, we have $\left\Vert \widehat{T}\right\Vert = \left\Vert \tilde{T}\right\Vert = ||T||_{Coh,p}$

\vspace{0,15cm}
\begin{definition}
Let $1 < p < \infty$ and $E,F$ Banach spaces. An $n$-homogeneous polynomial
$P \in\mathcal{P}(^{n}E;F)$ is Cohen strongly $p$-summing if
\[
\left(  P(x_{i})\right)  _{i=1}^{\infty}\in l_{p}\langle F\rangle
\ \ \mathrm{whenever}\ \ (x_{i})_{i=1}^{\infty}\in l_{np}(E) \ .
\]
\end{definition}

\vspace{0,15cm} The set of all $n$-homogeneous Cohen strongly $p$-summing
polynomials from $E$ to $F$ will be denoted by $\mathcal{P}_{Coh,p}(^{n}E;F)$. Is simple to
prove that $\mathcal{P}_{Coh,p}(^{n}E;F)$, provided with the usual operations
is a subspace of $\mathcal{P}(^{n}E;F)$. It follows directly from the
definition and the polarization formula (see \cite{Mu}) that a polynomial $P
\in\mathcal{P}(^{n}E;F)$ is Cohen strongly $p$-summing if and only if
$\check{P} \in\mathcal{L}_{s}(^{n}E;F)$ is Cohen strongly $p$-summing. The next result follows the lines of Proposition \ref{prop4}

\vspace{0,15cm}

\begin{proposition}
For $P \in\mathcal{P}(^{n}E;F)$ and $1/p + 1/p^{*} =1$, the following
statements are equivalent:

\vspace{0,15cm}
$(i)$ $P$ is Cohen strongly $p$-summing;

$(ii)$ there is a $C>0$ such that
\begin{equation*}
\sum_{i=1}^{\infty}|\varphi_{i}(P(x_{i}))| \leq C \left(  \sum_{i=1}^{\infty
}||x_{i}||^{np}\right)  ^{1/p}\, ||(\varphi_{i})_{i=1}^{\infty}||_{w,p^{*}}\ ,
\end{equation*}
whenever $\left(  x_{i}\right)  _{i=1}^{\infty}\in l_{np}(E)$ and
$(\varphi_{i})_{i=1}^{\infty}\in l_{p^{*}}^{w}(F^{^{\prime}})$.

$(iii)$ there is a $C>0$ such that
\begin{equation}
\label{polcoh}\sum_{i=1}^{m} |\varphi_{i}(P(x_{i}))| \leq C \left(  \sum
_{i=1}^{m} ||x_{i}||^{np} \right)  ^{1/p}\, ||(\varphi_{i})_{i=1}%
^{m}||_{w,p^{*}}\ ,
\end{equation}
for all $m \in\mathbb{N},\ x_{i} \in E,\ \varphi_{i} \in F^{^{\prime}%
},\ i=1,...,m$ .

Moreover, the smallest $C$ such that (\ref{polcoh}) is satisfied, denoted by
$||P||_{Coh,p}$ , defines a norm in $\mathcal{P}_{Coh,p}(^{n}E;F)$.
\end{proposition}

\begin{proof}
Since $P$ is Cohen strongly $p$-summing if and only if
$\check{P}$ is Cohen strongly $p$-summing all statements follow by using Proposition \ref{prop4} with $\check{P}$.
\end{proof}

\vspace{0,15cm} \noindent\textbf{An abstract result on multilinear Cohen
summing operators}

\vspace{0,20cm} By invoking Proposition \ref{prop4}, we can define a Cohen
strongly $p$-summing multilinear operator by using the statement $(iii)$, in
the same proposition. We mention that many authors, including D. Achour
\cite{achour10} and V. Dimant \cite{dimant03} define a Cohen strongly $p$-summing operator by:

\vspace{0,15cm}
\begin{definition}
[Achour, \cite{achour10}]\label{defx} Let $1 \leq p \leq\infty$, $m
\in\mathbb{N}$, $E_{j},F$ Banach spaces, $j=1,...,n$. A continuous $n$-linear operator $T:E_{1}
\times... \times E_{n} \rightarrow F$ is Cohen
strongly $p$-summing if there exist a constant $C>0$ such that for all $x_{1}%
^{(j)},...,x_{m}^{(j)} \in E_{j}$ and $\varphi_{1},...,\varphi_{m} \in
F^{^{\prime}}$
\begin{equation}
\label{cohen}\sum_{i=1}^{m} |\varphi_{i}(T(x_{i}^{(1)},...,x_{i}^{(n)}))| \leq C
\left(  \sum_{i=1}^{m} \prod_{j=1}^{n} ||x_{i}^{(j)}||^{p} \right)
^{1/p}\ ||(\varphi_{i})_{i=1}^{m}||_{w,p^{*}} \ .
\end{equation}

\end{definition}

\vspace{0.15cm}

With this definition, Achour \cite{achour07} proves a Pietsch Domination
Theorem (\ref{dom}) for this class of operators. Below we show that the
definitions given by (\ref{cohen}) and (\ref{def1}) are in fact equivalent
since both of them are characterized by the same Pietsch Domination Theorem. We stress that since
\[
\left(  \sum_{i=1}^{m} \left(  ||x_{i}^{(1)}||\ ... \ ||x_{i}^{(n)}||\right)
^{p} \right)  ^{1/p} \leq\left(  \sum_{i=1}^{m} ||x_{i}^{(1)}||^{np} \right)
^{1/np} ... \ \left(  \sum_{i=1}^{m} ||x_{i}^{(n)}||^{np} \right)  ^{1/np}\ ,%
\]
the implication (\ref{cohen}) $\Rightarrow$ (\ref{def1}) is obvious. However the implication (\ref{def1}) $\Rightarrow$ (\ref{cohen}) seems not straightforward. The main tool for the proof of this equivalence is the Full General Pietsch Domination Theorem \cite{BPRn, jmaa1, daniel12}.

\vspace{0,15cm} Let $X_{1},...,X_{n},Y$ and $E_{1},...,E_{r}$ arbitrary
non-empty sets, $\mathcal{H}$ a family of operators from $X_{1} \times
\cdots\times X_{n}$ to $Y$. Also be $K_{1},...,K_{t}$ compact Hausdorff
topological spaces, $G_{1},...,G_{t}$ Banach spaces and suppose that the
mappings
\[
\left\{
\begin{array}
[c]{l}%
R_{j}\ :\ K_{j} \times E_{1} \times\cdots\times E_{r} \times G_{j}
\rightarrow[0,\infty)\ , \ j=1,...,t\ ,\\
S\ :\ \mathcal{H} \times E_{1} \times\cdots\times E_{r} \times G_{1}
\times\cdots\times G_{t} \rightarrow[0,\infty)\\
\end{array}
\right.
\]
have the following properties:

\begin{description}
\item[$1$] for each $x^{(l)} \in E_{l}$ and $b \in G_{j}$, with $(j,l)
\in\{1,...,t\} \times\{1,...,r\}$, the mapping
\[
(R_{j})_{x^{(1)},...,x^{(r)},b}\ : \ K_{j} \rightarrow[0,\infty)\ ,
\]
defined by $(R_{j})_{x^{(1)},...,x^{(r)},b}(\varphi) = R_{j}(\varphi
,x^{(1)},...,x^{(r)},b)$, is continuous;

\item[$2$] the following inequalities hold:
\[
\left\{
\begin{array}
[c]{l}%
R_{j}(\varphi,x^{(1)},...,x^{(r)},\eta_{j} b^{(j)}) \leq\eta_{j} R_{j}%
(\varphi,x^{(1)},...,x^{(r)}, b^{(j)})\\
S(f, x^{(1)},...,x^{(r)}, \alpha_{1} b^{(1)},...,\alpha_{t} b^{(t)})
\geq\alpha_{1}...\alpha_{t} S(f, x^{(1)},...,x^{(r)}, b^{(1)},..., b^{(t)})\\
\end{array}
\right.
\]
for all $\varphi\in K_{j},\ x^{(l)} \in E_{l}$ (with $l=1,...,r$), $0 \leq
\eta_{j}, \alpha_{j} \leq1,\ b^{(j)} \in G_{j}$, $j=1,...,t$ and $f
\in\mathcal{H}$.
\end{description}

Under these conditions, we have the following definition and theorem taken from \cite{daniel12}:

\vspace{0,15cm}

\begin{definition}
Let $0<p_{1},...,p_{t},p_{0}<\infty$, with $\frac{1}%
{p_{0}}=\frac{1}{p_{1}}+\cdots+\frac{1}{p_{t}}$. An application $f\,:\,X_{1}%
\times\cdots\times X_{n}\rightarrow Y\in\mathcal{H}$ is $R_{1},...,R_{t}$%
-$S$-abstract $(p_{1},...,p_{t})$-summing if there is a $C>0$ such that
\[
\left(  \sum_{i=1}^{m}\left(  S(T,x_{i}^{(1)},...,x_{i}^{(r)},b_{i}%
^{(1)},...,b_{i}^{(t)})\right)  ^{p_{0}}\right)  ^{1/p_{0}}\leq C\prod
_{k=1}^{t}\sup_{\varphi\in K_{k}}\left(  \sum_{i=1}^{m}R_{k}(\varphi
,x_{i}^{(1)},...,x_{i}^{(r)},b_{i}^{(k)})^{p_{k}}\right)  ^{1/p_{k}}%
\]
for all $x_{1}^{(s)},...,x_{m}^{(s)}\in E_{s},\ b_{1}^{(s)},...,b_{m}^{(s)}\in
G_{l},\ m\in\mathbb{N}$ and $(s,l)\in\{1,...,r\}\times\{1,...,t\}$.
\end{definition}

\vspace{0,15cm}

\begin{theorem}[Full General PDT]
An application $f\in\mathcal{H}$ is $R_{1},...,R_{t}$%
-$S$-abstract $(p_{1},...,p_{t})$-summing if and only if there exist $C>0$ and
Borel probability measures $\mu_{k}$ in $K_{k},\ k=1,...,t$, such that
\[
S(T,x^{(1)},...,x^{(r)},b^{(1)},...,b^{(t)})\leq C\prod_{k=1}^{t}\left(
\int_{K_{k}}R_{k}(\varphi,x^{(1)},...,x^{(r)},b^{(k)})^{p_{k}}d\mu_{k}\right)
^{1/p_{k}}\ ,
\]
for all $x^{(l)}\in E_{l},\ l=1,...,r$ and $b^{(k)}\in G_{k}$, with
$k=1,...,t$.
\end{theorem}

\vspace{0,15cm} Now, we can prove the equivalence between the different approaches to Cohen multilinear operators:

\vspace{0,15cm}
\begin{theorem}\label{teo2}
Let $1 < p < \infty$ and $1/p + 1/p^* = 1$. For $T\in\mathcal{L}(X_{1},...,X_{n};Y)$, the following
statements are equivalent:

\vspace{0cm} $(i)$ there is a $C>0$ such that
\begin{equation*}
\sum_{i=1}^{m}|\varphi_{i}(T(x_{i}^{(1)},...,x_{i}^{(n)}))|\leq C\left(
\sum_{i=1}^{m}\prod_{j=1}^{n}||x_{i}^{(j)}||^{p}\right)  ^{1/p}\ ||(\varphi
_{i})_{i=1}^{m}||_{w,p^{\ast}},
\end{equation*}
for all $m\in\mathbb{N},\ x_{i}^{(j)}\in X_{j},\ \varphi_{i}\in Y^{^{\prime}%
},\ i=1,...,m\ ,\ j=1,...,n$ ;

$(ii)$ there is a $C>0$ such that
\begin{equation}
\label{defq}\sum_{i=1}^{m} |\varphi_{i}(T(x_{i}^{(1)},...,x_{i}^{(n)}))| \leq
C \prod_{j=1}^{n} \left(  \sum_{i=1}^{m} ||x_{i}^{(j)}||^{np} \right)  ^{1/np}
\ ||(\varphi_{i})_{i=1}^{m}||_{w,p^{*}}\ ,
\end{equation}

for all $m\in\mathbb{N},\ x_{i}^{(j)}\in X_{j},\ \varphi_{i}\in Y^{^{\prime}%
},\ i=1,...,m\ ,\ j=1,...,n$ ;

$(iii)$ there are a $C>0$ and a Borel probability measure $\mu$ in $B_{Y^{''}}$
such that
\begin{equation}
|\varphi(T(x_{1},...,x_{n}))|\leq C||x_{1}||\ ...\ ||x_{n}||\left(
\int_{B_{Y^{''}}}|\psi(\varphi)|^{p^{\ast}}d\mu(\psi)\right)  ^{1/{p^{\ast}}}\ ,
\label{dom}%
\end{equation}
for all $x_{j}\in X_{j},\ \varphi\in Y^{^{\prime}},\ j=1,...,n$ .
\end{theorem}

\begin{proof}
$(i)\Rightarrow(ii):$ As mentioned before, we just need to use H\"{o}lder's inequality.

$(ii)\Rightarrow(iii):$ Using the Full General PDT, choosing%

\[
\left\vert
\begin{array}
[c]{l}%
t=n+1\ \text{and}\ \ r=1\\
E_1 = \{0\}\\
K_k = \{0\}\ , k=1,...,n \ \text{and}\ \ K_{n+1} =B_{Y^{^{\prime\prime}}}\\
G_{k}=X_{k},\ \ k=1,...,n,\ \mathrm{and}\ \ G_{n+1}=Y^{^{\prime}}\\
\mathcal{H}=\mathcal{L}(X_{1},...,X_{n};Y)\\
p_{0}=1,\ \ p_{k}=np,\ \ k=1,...,n\ \mathrm{and}\ \ p_{n+1}=p^{\ast}\\
S(T,0,x^{(1)},...,x^{(n)},\varphi)=|\varphi(T(x^{(1)},...,x^{(n)}))|\\
R_{k}(\gamma,0,x^{(k)})=||x^{(k)}||,\ k=1,...,n\\
R_{n+1}(\psi,0,\varphi)=|\psi(\varphi)|\\
\end{array}
\right.
\]
we have
\begin{align*}
\left(  \sum_{i=1}^{m}\left(  S(T,x_{i}^{(1)},...,x_{i}^{(r)},b_{i}%
^{(1)},...,b_{i}^{(t)})\right)  ^{p_{0}}\right)  ^{1/p_{0}}  &  =\left(
\sum_{i=1}^{m}\left(  S(T,0,x_{i}^{(1)},...,x_{i}^{(n)}%
,\varphi_{i})\right)  ^{p_{0}}\right)
^{1/p_{0}}\\
&  =\sum_{i=1}^{m}|\varphi_{i}(T(x_{i}^{(1)},...,x_{i}^{(n)}))|
\end{align*}
and also
\begin{align*}
&  \prod_{k=1}^{n+1}\sup_{\varphi\in K_{k}}\left(  \sum_{i=1}^{m}R_{k}%
(\varphi,x_{i}^{(1)},...,x_{i}^{(r)},b_{i}^{(k)})^{p_{k}}\right)  ^{1/p_{k}%
}=\\
&  =\sup_{\varphi\in K_{n+1}}\left(  \sum_{i=1}^{m}R_{n+1}(\psi,0,
\varphi_{i})^{p_{n+1}}\right)  ^{1/p_{n+1}}\cdot
\prod_{k=1}^{n}\sup_{\varphi\in K_{k}}\left(  \sum_{i=1}^{m}R_{k}(\gamma
,0,x_{i}^{(k)})^{p_{k}}\right)  ^{1/p_{k}}\\
&  =\sup_{\psi\in B_{Y^{^{\prime\prime}}}}\left(  \sum_{i=1}^{m}|\psi\left(
\varphi_{i}\right)  |^{p^{\ast}}\right)  ^{1/{p^{\ast}}}\cdot\prod_{k=1}%
^{n}\left(  \sum_{i=1}^{m}||x_{i}^{(k)}||^{np}\right)  ^{1/np}\ \\
&  =||(\varphi_{i})_{i=1}^{m}||_{w,p^{\ast}}\cdot\ \prod_{k=1}^{n}\left(
\sum_{i=1}^{m}||x_{i}^{(k)}||^{np}\right)  ^{1/np}\ .
\end{align*}

Thus, $T$ satisfies (\ref{defq}) if and only if is $R_{1},...,R_{t}$%
-$S$-abstract $(p_{1},...,p_{t})$-summing. By the Full General PDT there are a constant $C>0$ and Borel probability measures $\mu_{k}$ in
$K_{k},\ k=1,...,t$ such that
\[
S(T,x_{1},...,x_{r},b_{1},...,b_{t})\leq C\prod_{k=1}^{t}\left(  \int_{K_{k}%
}R_{k}(\varphi,x_{1},...,x_{r},b_{k})^{p_{k}}d\mu_{k}\right)  ^{1/p_{k}}\ ,
\]
that is,
\begin{align*}
|\varphi(T(x_{1},...,x_{n}))| &  \leq C\left[  \prod_{k=1}^{n}\left(
\int_{K_{k}}||x_{k}||^{np}d\mu_{k}\right)  ^{1/np}\right]
\ \left(  \int_{B_{Y^{^{\prime\prime}}}}|\psi(\varphi)|^{p^{\ast}}d\mu
(\psi)\right)  ^{1/{p^{\ast}}}\\
&  \leq C||x_{1}||\ ...\ ||x_{n}||\left(  \int_{B_{Y^{^{\prime\prime}}}}%
|\psi(\varphi)|^{p^{\ast}}d\mu(\psi)\right)  ^{1/{p^{\ast}}}\ .
\end{align*}

$(iii)\Rightarrow(i):$ (Theorem 2.4 in \cite{achour07}) For all $m\in
\mathbb{N}$, if $1\leq i\leq m$ we have, by (\ref{dom}),
\[
|\varphi_{i}(T(x_{i}^{(1)},...,x_{i}^{(n)}))|\leq C||x_{i}^{(1)}%
||\ ...\ ||x_{i}^{(n)}||\left(  \int_{B_{Y^{^{\prime\prime}}}}|\psi
(\varphi_{i})|^{p^{\ast}}d\mu(\psi)\right)  ^{1/{p^{\ast}}}\ .
\]
Then,
\begin{align*}
\sum_{i=1}^{m}|\varphi_{i}(T(x_{i}^{(1)},...,x_{i}^{(n)}))| &  \leq
C\sum_{i=1}^{m}\left(  ||x_{i}^{(1)}||\ ...\ ||x_{i}^{(n)}||\left(
\int_{B_{Y^{^{\prime\prime}}}}|\psi(\varphi_{i})|^{p^{\ast}}d\mu(\psi)\right)
^{1/{p^{\ast}}}\right)  \\
&  \leq C\left(  \sum_{i=1}^{m}\left(  ||x_{i}^{(1)}||\ ...\ ||x_{i}%
^{(n)}||\right)  ^{p}\right)  ^{1/p}\left(  \sum_{i=1}^{m}\left(
\int_{B_{Y^{^{\prime\prime}}}}|\psi(\varphi_{i})|^{p^{\ast}}d\mu(\psi)\right)
\right)  ^{1/{p^{\ast}}}\\
&  =C\left(  \sum_{i=1}^{m}\left(  ||x_{i}^{(1)}||\ ...\ ||x_{i}%
^{(n)}||\right)  ^{p}\right)  ^{1/p}\left(  \int_{B_{Y^{^{\prime\prime}}}}%
\sum_{i=1}^{m}|\psi(\varphi_{i})|^{p^{\ast}}d\mu(\psi)\right)  ^{1/{p^{\ast}}%
}\\
&  \leq C\left(  \sum_{i=1}^{m}\left(  ||x_{i}^{(1)}||\ ...\ ||x_{i}%
^{(n)}||\right)  ^{p}\right)  ^{1/p}\left(  \sup_{\psi\in B_{Y^{^{\prime
\prime}}}}\sum_{i=1}^{m}|\psi(\varphi_{i})|^{p^{\ast}}\right)  ^{1/{p^{\ast}}%
}\\
&  =C\left(  \sum_{i=1}^{m}\left(  ||x_{i}^{(1)}||\ ...\ ||x_{i}%
^{(n)}||\right)  ^{p}\right)  ^{1/p}\,||(\varphi_{i})_{i=1}^{m}||_{w,p^{\ast}%
}\ .
\end{align*}
\end{proof}

\vspace{0,15cm}
The previous theorem can be seen as a particular case of the next more general and abstract theorem. The tool used in the proof, again, is the Full General PDT (see \cite{daniel12}).

\vspace{0,15cm}
\begin{theorem}\label{teoabs}
Let $f:X_{1} \times\cdots\times X_{n} \rightarrow Y$ a mapping in
$\mathcal{H}$ and let
\[
0<p^{*},u,s,p_{1},...,p_{t-1},q_{1},...,q_{t-1}<\infty\ ,
\]
such that
\[
\frac{1}{u} = \frac{1}{p_{1}} + \cdots+ \frac{1}{p_{t-1}} + \frac{1}{p^{*}%
}\ \ \mathrm{and}\ \ \frac{1}{s} = \frac{1}{q_{1}} + \cdots+ \frac{1}{q_{t-1}}
+ \frac{1}{p^{*}}.
\]
If $R_{k_{(x_{1},...,x_{r},b)}}(\cdot)$ is constant, for each $x_{1},...,x_{r},b$,
then the following are equivalent:

\vspace{0,15cm} $(i)$ $f$ is $R_{1},...,R_{t}$-$S$-abstract $(p_{1}%
,...,p_{t-1},p^{*})$-summing;

$(ii)$ $f$ is $R_{1},...,R_{t}$-$S$-abstract $(q_{1},...,q_{t-1},p^{*})$-summing;
\end{theorem}

\begin{proof}
By the Full General PDT, $f$ is $R_{1},...,R_{t}$-$S$-abstract $(p_{1}%
,...,p_{t-1},p^{*})$-summing if and only if there are a constant $C$ and Borel probability measures
$\mu_{i}$ in $K_{i}$, $i=1,...,t$ , such that
\[
S(f,x_{1},...,x_{r},b_{1},...,b_{t}) \leq C \prod_{i=1}^{t} \left(
\int_{K_{i}} R_{i}(\varphi,x_{1},...,x_{r},b_{i})^{p_{i}}d\mu_{i}\right)
^{1/p_{i}}\ .
\]
that is, in our case, $f$ is $R_{1},...,R_{t}$-$S$-abstract $(p_{1}%
,...,p_{t-1},p^{*})$-summing if and only if there are a constant $C$ and a Borel probability measure $\mu$
in $K_{t}$ such that
\begin{equation}
\label{domin}S(f,x_{1},...,x_{r},b_{1},...,b_{t}) \leq C \left( \prod_{i=1}^{t-1}
R_{i}(\varphi,x_{1},...,x_{r},b_{i}) \right) \cdot\left(  \int_{K_{t}} R_{t}%
(\varphi,x_{1},...,x_{r},b_{t})^{p^{*}}d\mu\right)  ^{1/p^{*}}\ ,
\end{equation}
since, by hypothesis, for any fixed $\varphi_i \in K_i$
\[
\left(  \int_{K_{i}} R_{i}(\varphi,x_{1},...,x_{r},b_{i})^{p_{i}}d\mu
_{i}\right)  ^{1/p_{i}} = R_{i}(\varphi,x_{1},...,x_{r},b_{i}%
)\ ,\ i=1,...,t-1.
\]
On the other hand, the same reasoning shows that $f$ is $R_{1},...,R_{t}$%
-$S$-abstract $(q_{1},...,q_{t-1},p^{*})$-summing if and only if there are a constant $C$
and a Borel probability measure $\mu$ in $K_{t}$ such that
\[
S(f,x_{1},...,x_{r},b_{1},...,b_{t}) \leq C \left( \prod_{i=1}^{t-1} R_{i}%
(\varphi,x_{1},...,x_{r},b_{i})\right) \cdot\left(  \int_{K_{t}} R_{t}(\varphi
,x_{1},...,x_{r},b_{t})^{p^{*}}d\mu\right)  ^{1/p^{*}}\ ,
\]
expression that corresponds exactly to that given by (\ref{domin}).
\end{proof}

\vspace{0,15cm}
It is worth mentioning a somewhat surprising consequence of Theorem \ref{teoabs} for
linear operators. Let $p^* \in \left(  1,\infty\right)  $ be fixed and
\[
\Gamma=\left\{  \left(  r,q\right)  \in \left.\left[ 1,\infty \right.\right) \times  \left( 1,\infty\right) :\frac
{1}{r}= \frac{1}{q} + \frac{1}{p^{\ast}}\right\}  .
\]
Let $C_{r,q}(E;F)$ be the class of all $T\in\mathcal{L}\left(  E;F\right)  $
so that there is a $C \geq 0$ satisfying
\[
\left(\sum\limits_{j=1}^{m} \left\vert \varphi_{i}\left(  T\left(  x_{i}\right)  \right)  \right\vert
^{r}\right)  ^{1/r}\leq C\left\Vert \left(  x_{i}\right)  _{i=1}
^{m}\right\Vert _{q}\left\Vert \left(  \varphi_{i}\right)  _{i=1}
^{m}\right\Vert _{w,p^{\ast}}
\]
for all $m.$ From Theorem \ref{teoabs} it follows that
\[
C_{r_{1},q_{1}}(E;F)=C_{r_{2},q_{2}}(E;F)
\]
for all $\left(  r_{1},q_{1}\right)  ,\left(  r_{2}.q_{2}\right)  \in\Gamma.$ In particular,
\[
C_{r,q}(E;F)=\mathcal{D}_{p}(E;F)\ ,
\]
with $1 = 1/p + 1/p^*$, for all $\left(  r,q\right) \in\Gamma$.

\vspace{0,15cm}

\section{Multiple Cohen strongly summing multilinear operators}\label{omc}

In this Section we use the successful theory of multiple summing operators (see \cite{ma111, daniel12, perez03}) as a prototype to motivate the forthcoming notion of multiple Cohen strongly $p$-summing operators. As it will be clear along the paper, this new approach is adequate from the viewpoint of polynomial/multilinear ideals and holomorphy.

\vspace{0,15cm}

\begin{definition}
Let $1 < p < \infty$ and $E_{i},F$ be Banach spaces,
$i=1,...,n$, with $1/p + 1/p^{*} =1$. An operator $T \in\mathcal{L}%
(E_{1},...,E_{n};F)$ is multiple Cohen strongly $p$-summing if
\[
\left(  T\left(  x_{j_{1}}^{(1)},...,x_{j_{n}}^{(n)}\right)  \right)
_{j_{1},...,j_{n} \in\mathbb{N}} \in l_{p}\langle F \rangle\ ,\ \ \mathrm{for}\ \mathrm{any}\ \left(  x_{j}^{(i)}\right)  _{j=1}^{\infty}\in l_{p}(E_{i}),\ \ i=1,...,n.
\]
\end{definition}

\vspace{0,15cm}

\begin{remark}
In the above definitions and hereafter, we are identifying the elements of the
set $\mathbb{N}^{n}$ with elements in $\mathbb{N}$, to avoid overloaded
notations. Thus, instead of denoting, for example $a_{j_{1},...,j_{n}} \in
l_{1}(F;\mathbb{N}^{n})$ we write only $a_{j_{1},...,j_{n}} \in l_{1}(F)$.
\end{remark}

\vspace{0,15cm} The class of all multiple Cohen strongly
$p$-summing multilinear mappings is a subspace of $\mathcal{L}(E_{1},...,E_{n};F)$ (it is
easy to show) and will be denoted by $\mathcal{L}_{mCoh,p}(E_{1},...,E_{n};F)$.

\vspace{0,15cm}

\begin{proposition}
\label{mCohnorm} For $T \in\mathcal{L}(E_{1},...,E_{n};F)$ and $1/p + 1/p^{*}
=1$, the following statements are equivalent:

\vspace{0,15cm} $(i)$ $T$ is multiple Cohen strongly $p$-summing;

$(ii)$ there is a $C>0$ such that
\begin{align*}
&  \sum_{j_{1},...,j_{n}=1}^{\infty}\left\vert \varphi_{j_{1},...,j_{n}%
}\left(  T\left(  x_{j_{1}}^{(1)},...,x_{j_{n}}^{(n)}\right)  \right)
\right\vert \\
&  \leq C \left\Vert \left(  x_{j}^{(1)} \right)  _{j=1}^{\infty}\right\Vert
_{p} \cdots\left\Vert \left(  x_{j}^{(n)} \right)  _{j=1}^{\infty}\right\Vert
_{p} \left\Vert \left(  \varphi_{j_{1},...,j_{n}}\right)  _{j_{1},...,j_{n}
\in\mathbb{N}} \right\Vert _{w,p^{*}} \ ,
\end{align*}
for any $(\varphi_{j_{1},...,j_{n}})_{j_{1},...,j_{n} \in\mathbb{N}} \in
l_{p^{*}}^{w}(F^{^{\prime}})$ and $\left(  x_{j}^{(i)}\right)  _{j=1}^{\infty
}\in l_{p}(E_{i})$, $i=1,...,n$.

$(iii)$ there is a $C>0$ such that
\begin{align}
\label{desmCohen} &  \sum_{j_{1},...,j_{n}=1}^{m} \left\vert \varphi
_{j_{1},...,j_{n}}\left(  T\left(  x_{j_{1}}^{(1)},...,x_{j_{n}}^{(n)}\right)
\right)  \right\vert \\
&  \leq C \left\Vert \left(  x_{j}^{(1)} \right)  _{j=1}^{m} \right\Vert _{p}
\cdots\left\Vert \left(  x_{j}^{(n)} \right)  _{j=1}^{m} \right\Vert _{p}
\left\Vert \left(  \varphi_{j_{1},...,j_{n}}\right)  _{j_{1},...,j_{n}=1}^{m}
\right\Vert _{w,p^{*}} \ ,\nonumber
\end{align}
for all $m \in\mathbb{N}, \ \varphi_{j_{1},...,j_{n}} \in F^{^{\prime}}$ and
$x_{j}^{(i)} \in E_{i}$, $i=1,...,n, \ j_{i}=1,...,m, \ j=1,...,m$.

In addition, the smallest of the constants $C$ satisfying (\ref{desmCohen}),
denoted by $||T||_{mCoh,p}$ , defines a norm in $\mathcal{L}_{mCoh,p}%
(E_{1},...,E_{n};F)$.
\end{proposition}

\begin{proof}
$(i)\Rightarrow(ii)$ Since $T$ is multiple Cohen strongly $p$-summing, the
operator
\[
\widetilde{T}:\ l_{p^{\ast}}^{w}(F^{^{\prime}})\times l_{p}(E_{1})\times
\cdots\times l_{p}(E_{n})\longrightarrow l_{1}%
\]
given by
\[
\left(  (\varphi_{j_{1},...,j_{n}})_{j_{1},...,j_{n}\in\mathbb{N}},\left(
x_{j_{1}}^{(1)}\right)  _{j_{1}=1}^{\infty},...,\left(  x_{j_{n}}%
^{(n)}\right)  _{j_{n}=1}^{\infty}\right)  \longmapsto\left(  \varphi
_{j_{1},...,j_{n}}\left(  T\left(  x_{j_{1}}^{(1)},...,x_{j_{n}}^{(n)}\right)
\right)  \right)  _{j_{1},...,j_{n}\in\mathbb{N}}%
\]
is well-defined and is $(n+1)$-linear.

Let $(x_{k})_{k=1}^{\infty}\in l_{p^{\ast}}^{w}(F^{^{\prime}})\times
l_{p}(E_{1})\times\cdots\times l_{p}(E_{n})$ with
\begin{equation}
\left\{
\begin{array}
[c]{l}%
x_{k}\rightarrow x\in l_{p^{\ast}}^{w}(F^{^{\prime}})\times l_{p}(E_{1}%
)\times\cdots\times l_{p}(E_{n})\\
\widetilde{T}(x_{k})\rightarrow(z_{j_{1},...,j_{n}})_{j_{1},...,j_{n}%
\in\mathbb{N}}\in l_{1}%
\end{array}
\right.  .\label{pa1}%
\end{equation}
Writing
\begin{equation}
\left\{
\begin{array}
[c]{l}%
x_{k}=\left(  (\varphi_{j_{1},...,j_{n}}^{k})_{j_{1},...,j_{n}\in\mathbb{N}%
},\left(  x_{k,j_{1}}^{(1)}\right)  _{j_{1}=1}^{\infty},...,\left(
x_{k,j_{n}}^{(n)}\right)  _{j_{n}=1}^{\infty}\right)  \\
x=\left(  (\varphi_{j_{1},...,j_{n}})_{j_{1},...,j_{n}\in\mathbb{N}},\left(
x_{j_{1}}^{(1)}\right)  _{j_{1}=1}^{\infty},...,\left(  x_{j_{n}}%
^{(n)}\right)  _{j_{n}=1}^{\infty}\right)
\end{array}
\right.  \ ,\label{pa2}%
\end{equation}
we have
\begin{align*}
(z_{j_{1},...,j_{n}})_{j_{1},...,j_{n}\in\mathbb{N}} &  =\lim_{k\rightarrow
\infty}\widetilde{T}(x_{k}) \\
&  =\lim_{k\rightarrow\infty}\left(  \varphi_{j_{1},...,j_{n}}^{k}\left(
T\left(  x_{k,j_{1}}^{(1)},...,x_{k,j_{n}}^{(n)}\right)  \right)  \right)
_{j_{1},...,j_{n}\in\mathbb{N}}
\end{align*}

We need to show that%
\[
\widetilde{T}(x)=(z_{j_{1},...,j_{n}})_{j_{1},...,j_{n}\in\mathbb{N}}.
\]
We have%
\begin{align*}
\widetilde{T}(x)  & =\widetilde{T}\left(  \left(  (\varphi_{j_{1},...,j_{n}%
})_{j_{1},...,j_{n}\in\mathbb{N}},\left(  x_{j_{1}}^{(1)}\right)  _{j_{1}%
=1}^{\infty},...,\left(  x_{j_{n}}^{(n)}\right)  _{j_{n}=1}^{\infty}\right)
\right)  \\
& =\left(  \varphi_{j_{1},...,j_{n}}\left(  T\left(  x_{j_{1}}^{(1)}%
,...,x_{j_{n}}^{(n)}\right)  \right)  \right)  _{j_{1},...,j_{n}\in\mathbb{N}%
}.
\end{align*}
So we need to show that%
\begin{equation}
\varphi_{j_{1},...,j_{n}}\left(  T\left(  x_{j_{1}}^{(1)},...,x_{j_{n}}%
^{(n)}\right)  \right)  =z_{j_{1},...,j_{n}}\label{pop}%
\end{equation}
for all $j_{1},...,j_{n}\in\mathbb{N}$.

But
\begin{equation}
\lim_{k\rightarrow\infty}\varphi_{j_{1},...,j_{n}}^{k}\left(  T\left(
x_{k,j_{1}}^{(1)},...,x_{k,j_{n}}^{(n)}\right)  \right)  =z_{j_{1},...,j_{n}%
}\ \label{pa4}%
\end{equation}
and, on the other hand, from (\ref{pa1}) and (\ref{pa2}) it follows that
\begin{equation}
x_{k,j}^{(i)}\rightarrow x_{j}^{(i)}\ ,\ \mathrm{in}\ E_{i}\ \ \mathrm{and}%
\ \ \varphi_{j_{1},...,j_{n}}^{k}\rightarrow\varphi_{j_{1},...,j_{n}%
}\ \mathrm{in}\ F^{^{\prime}}\ ,\label{pa5}%
\end{equation}
for any $j\in\mathbb{N},\ i=1,...,n$ and $j_{1},...,j_{n}\in\mathbb{N}$. Since
$T$ is continuous, it follows from (\ref{pa5}) that%
\begin{equation}
\lim_{k\rightarrow\infty}\varphi_{j_{1},...,j_{n}}^{k}\left(  T\left(
x_{k,j_{1}}^{(1)},...,x_{k,j_{n}}^{(n)}\right)  \right)  =\varphi
_{j_{1},...,j_{n}}\left(  T\left(  x_{j_{1}}^{(1)},...,x_{j_{n}}^{(n)}\right)
\right)  \ ,\label{pa6}%
\end{equation}
for all $j_{1},...,j_{n}\in\mathbb{N}$. So, from (\ref{pa6}) and (\ref{pa4})
we obtain (\ref{pop}). Thus, $T$ has closed graph and hence is continuous. Thus we have
\begin{align*}
&  \sum_{j_{1},...,j_{n}=1}^{\infty}\left\vert \varphi_{j_{1},...,j_{n}%
}\left(  T\left(  x_{j_{1}}^{(1)},...,x_{j_{n}}^{(n)}\right)  \right)
\right\vert \\
&  = \left\Vert \varphi_{j_{1},...,j_{n}}\left(  T\left(  x_{j_{1}}%
^{(1)},...,x_{j_{n}}^{(n)}\right)  \right)  _{j_{1},...,j_{n} \in\mathbb{N}}
\right\Vert _{1}\\
&  = \left\Vert \widetilde{T}\left(  (\varphi_{j_{1},...,j_{n}})_{j_{1}%
,...,j_{n} \in\mathbb{N}}, \left(  x_{j_{1}}^{(1)}\right)  _{j_{1}=1}^{\infty
},...,\left(  x_{j_{n}}^{(n)}\right)  _{j_{n}=1}^{\infty}\right)  \right\Vert
_{1}\\
&  \leq\left\Vert \widetilde{T}\right\Vert \left\Vert \left(  x_{j}^{(1)}
\right)  _{j=1}^{\infty}\right\Vert _{p} \cdots\left\Vert \left(  x_{j}^{(n)}
\right)  _{j=1}^{\infty}\right\Vert _{p} \left\Vert \left(  \varphi
_{j_{1},...,j_{n}}\right)  _{j_{1},...,j_{n} \in\mathbb{N}} \right\Vert
_{w,p^{*}}\ .
\end{align*}

$(ii) \Rightarrow(i)$ and $(ii) \Rightarrow(iii)$ are immediate.

$(iii) \Rightarrow(ii)$ Let $(\varphi_{j_{1},...,j_{n}})_{j_{1},...,j_{n}
\in\mathbb{N}} \in l_{p^{*}}^{w}(F^{^{\prime}})$ and $\left(  x_{j}%
^{(i)}\right)  _{j=1}^{\infty}\in l_{p}(E_{i})$, $i=1,...,n$. Then,%

\begin{align*}
&  \sum_{j_{1},...,j_{n}=1}^{\infty}\left\vert \varphi_{j_{1},...,j_{n}%
}\left(  T\left(  x_{j_{1}}^{(1)},...,x_{j_{n}}^{(n)}\right)  \right)
\right\vert = \sup_{m} \left(  \sum_{j_{1},...,j_{n}=1}^{m} \left\vert
\varphi_{j_{1},...,j_{n}}\left(  T\left(  x_{j_{1}}^{(1)},...,x_{j_{n}}%
^{(n)}\right)  \right)  \right\vert \right) \\
&  \leq\sup_{m} \left(  C \left\Vert \left(  x_{j}^{(1)} \right)  _{j=1}^{m}
\right\Vert _{p} \cdots\left\Vert \left(  x_{j}^{(n)} \right)  _{j=1}^{m}
\right\Vert _{p} \left\Vert \left(  \varphi_{j_{1},...,j_{n}}\right)
_{j_{1},...,j_{n}=1}^{m} \right\Vert _{w,p^{*}} \right) \\
&  = C \left\Vert \left(  x_{j}^{(1)} \right)  _{j=1}^{\infty}\right\Vert _{p}
\cdots\left\Vert \left(  x_{j}^{(n)} \right)  _{j=1}^{\infty}\right\Vert _{p}
\left\Vert \left(  \varphi_{j_{1},...,j_{n}}\right)  _{j_{1},...,j_{n}
\in\mathbb{N}} \right\Vert _{w,p^{*}}\ .
\end{align*}

\end{proof}

\vspace{0,15cm} The following result shows that the definition of multiple
Cohen strongly $p$-summing operator encompasses the concept of Cohen strongly
$p$-summing operators.

\vspace{0,15cm}

\begin{proposition}
\label{inclusaom} Every Cohen strongly $p$-summing multilinear operator is
multiple Cohen strongly $p$-summing and $||\cdot||_{mCoh,p} \leq
||\cdot||_{Coh,p}$.
\end{proposition}

\begin{proof}
By the Pietsch Domination Theorem, $T \in\mathcal{L}_{Coh,p}(E_{1}%
,...,E_{n};F)$ if and only if there are a constant $C$ and a Borel probability measure $\mu$ in $B_{F^{''}}$
such that
\begin{equation*}
|\varphi(T(x_{1},...,x_{n}))| \leq C ||x_{1}||\ ...\ ||x_{n}|| \left(
\int_{B_{F^{''}}}|\psi(\varphi)|^{p^{*}} d\mu(\psi) \right)  ^{1/{p^{*}}}\ ,
\end{equation*}
for all $x_{j} \in E_{j},\ \varphi\in F^{^{\prime}},\ j=1,...,n$ .

Thus, given $m \in\mathbb{N}$, we have
\[
|\varphi_{j_{1},...,j_{n}}(T(x_{j_{1}}^{(1)},...,x_{j_{n}}^{(n)}))| \leq C
||x_{j_{1}}^{(1)}||\ ...\ ||x_{j_{n}}^{(n)}|| \left(  \int_{B_{F^{''}}}%
|\psi(\varphi_{j_{1},...,j_{n}})|^{p^{*}} d\mu(\psi) \right)  ^{1/{p^{*}}}\ ,
\]
for all $\varphi_{j_{1},...,j_{n}} \in F^{^{\prime}}$ and $x_{j}^{(i)} \in
E_{i}$, $i=1,...,n$, $1 \leq j_{1},...,j_{n} \leq m$. Then,
\begin{align*}
&  \sum_{j_{1},...,j_{n}=1}^{m} |\varphi_{j_{1},...,j_{n}}(T(x_{j_{1}}%
^{(1)},...,x_{j_{n}}^{(n)}))|\\
&  \leq C \sum_{j_{1},...,j_{n}=1}^{m} \left(  ||x_{j_{1}}^{(1)}%
||\ ...\ ||x_{j_{n}}^{(n)}|| \left(  \int_{B_{F^{''}}}|\psi%
(\varphi_{j_{1},...,j_{n}})|^{p^{*}} d\mu(\psi) \right)  ^{1/{p^{*}}}\right) \\
&  \leq C \left(  \sum_{j_{1},...,j_{n}=1}^{m} \left(  ||x_{j_{1}}%
^{(1)}||\ ...\ ||x_{j_{n}}^{(n)}|| \right)  ^{p}\right)  ^{1/p} \left(
\sum_{j_{1},...,j_{n}=1}^{m} \int_{B_{F^{''}}}|\psi%
(\varphi_{j_{1},...,j_{n}})|^{p^{*}} d\mu(\psi) \right)  ^{1/p^{*}}\\
&  = C \left(  \sum_{j=1}^{m} ||x_{j}^{(1)}||^{p} \right)  ^{1/p}
\cdots\left(  \sum_{j=1}^{m} ||x_{j}^{(n)}||^{p} \right)  ^{1/p} \left(
\int_{B_{F^{''}}} \sum_{j_{1},...,j_{n}=1}^{m} |\psi%
(\varphi_{j_{1},...,j_{n}})|^{p^{*}} d\mu(\psi) \right)  ^{1/p^{*}}\\
&  \leq C \left\Vert \left(  x_{j}^{(1)} \right)  _{j=1}^{m} \right\Vert _{p}
\cdots\left\Vert \left(  x_{j}^{(n)} \right)  _{j=1}^{m} \right\Vert _{p}
\left(  \sup_{\psi \in B_{F^{''}}} \sum_{j_{1},...,j_{n}=1}^{m} |\psi(\varphi
_{j_{1},...,j_{n}})|^{p^{*}} \right)  ^{1/p^{*}}\\
&  = C \left\Vert \left(  x_{j}^{(1)} \right)  _{j=1}^{m} \right\Vert _{p}
\cdots\left\Vert \left(  x_{j}^{(n)} \right)  _{j=1}^{m} \right\Vert _{p}
\left\Vert \left(  \varphi_{j_{1},...,j_{n}}\right)  _{j_{1},...,j_{n}=1}^{m}
\right\Vert _{w,p^{*}} \ ,
\end{align*}
and thus $T \in\mathcal{L}_{mCoh,p}(E_{1},...,E_{n};F)$.
\end{proof}

\vspace{0,15cm}
An immediate consequence of Theorem 2.2.2 in Cohen's article
(\cite{cohen73}, page 184) is that the Dvoretzky-Rogers theorem is valid for Cohen
strongly $p$-summing linear operators:

\vspace{0,15cm}
\begin{theorem} If $E$ is a Banach space
then $id_{E}\,: E \rightarrow E$ is Cohen strongly $p$-summing if and only if
$\mathrm{dim}\,E < \infty$.
\end{theorem}

\vspace{0,15cm}
Moreover, by invoking the property (CP1) in Definition \ref{defcomp}, which is
satisfied by the class $\mathcal{L}_{mCoh,p}$, we conclude that if
$u:\,E\rightarrow F$ is not Cohen strongly $p$-summing, then the operator
\[
\psi\,:E\times\cdots\times E\rightarrow F,\ \ \mathrm{defined}\ \mathrm{by}%
\ \ \psi(x_{1},...,x_{n})=\varphi(x_{1})\ldots\varphi(x_{n-1})u(x_{n})\ ,
\]
where $0\neq\varphi\in E^{^{\prime}}$, does not belong to $\mathcal{L}%
_{mCoh,p}$. This shows that the class of multiple Cohen strongly $p$-summing
multilinear operators, though it contains the class $\mathcal{L}_{Coh,p},$ is
not so large.

\vspace{0,15cm}

\section{Linear, polynomial and multilinear ideals of Cohen strongly summing operators}

\label{ic}

Now we present some concepts and results on the theory of operators
ideals aiming to study the ideals of Cohen strongly $p$-summing
operators/polynomials/multilinear operators. The proofs are omitted; the
book \cite{pietsch80} is an excellent reference.

\vspace{0,15cm}

\begin{definition}
An operator ideal $\mathcal{I}$ is a subclass of the class $\mathcal{L}$ of
all continuous linear operators between Banach spaces such that for any Banach
spaces $E$ and $F,$ the components $\mathcal{I}\left(  E;F\right)
=\mathcal{L}\left(  E;F\right)  \cap\mathcal{I}$ satisfy:

$(i)$ $\mathcal{I}\left(  E,F\right)  $ is a linear subspace of
$\mathcal{L}\left(  E;F\right)  $ that contains the finite rank operators;

$(ii)$ The ideal property: if $u\in\mathcal{L}\left(  E;F\right)  $,
$v\in\mathcal{I}\left(  F,G\right)  $ and $t\in\mathcal{L}\left(  G,H\right)
,$ then $tvu\in\mathcal{I}\left(  E;H\right)  .$
\end{definition}

\vspace{0,15cm}

\begin{definition}
A normed ($p$-normed) operator ideal $\left(  \mathcal{I},\left\Vert
\cdot\right\Vert _{\mathcal{I}}\right)  $ is an operator ideal $\mathcal{I}$
equipped with a mapping $\left\Vert \cdot\right\Vert _{\mathcal{I}%
}:\mathcal{I}\rightarrow\left[  0,\infty\right)  $ such that:

$(i)$ $\left\Vert \cdot\right\Vert _{\mathcal{I}}$ constrained to
$\mathcal{I}\left(  E;F\right)  $ is a norm ($p$-norm) for any Banach spaces
$E$ and $F$;

$(ii)$ $\left\Vert id_{\mathbb{K}}\right\Vert _{\mathcal{I}}=1,$ with
$id_{\mathbb{K}}:\mathbb{K\rightarrow K}$ given by $id_{\mathbb{K}}\left(
x\right)  =x;$

$(iii)$ If $u\in\mathcal{L}\left(  E,F\right)  $ ,$v\in\mathcal{I}\left(
F;G\right)  $ and $t\in\mathcal{L}\left(  G;H\right)  $ , then $\left\Vert
tvu\right\Vert _{\mathcal{I}}\leq\left\Vert t\right\Vert \left\Vert
v\right\Vert _{\mathcal{I}}\left\Vert u\right\Vert .$
\end{definition}

\vspace{0,15cm} When the ideal components $\mathcal{I}\left(  E;F\right)  $ are
complete with respect to $\left\Vert \cdot\right\Vert _{\mathcal{I}},$ we say
that $\mathcal{I}$ is a complete ideal (or Banach ideal). In the case of
$p$-normed ideals, we say that $\mathcal{I}$ is a quasi-Banach ideal. The
inequality $\left\Vert t\right\Vert \leq\left\Vert t\right\Vert _{\mathcal{I}%
}$ for any $t\in\mathcal{I}$, valid for any ideal, is a very useful property
in the statements of completeness.

\vspace{0,15cm}

\begin{definition}
A multilinear operator ideal $\mathcal{M}$ is a subclass of the all class of
continuous multilinear operators between Banach such that for any
$n\in\mathbb{N}$ and Banach spaces $E_{1},...,E_{n}$ and $F$, the components
$\mathcal{M}\left(  E_{1},...,E_{n};F\right)  =\mathcal{L}\left(
E_{1},...,E_{n};F\right)  \cap\mathcal{M}$ satisfy:

$\left(  i\right)  $ $\mathcal{M}\left(  E_{1},...,E_{n};F\right)  $ is a
subspace of $\mathcal{L}\left(  E_{1},...,E_{n};F\right)  $ which contains the
$n$-linear finite type operators;

$\left(  ii\right)  $ The ideal property: if $A\in\mathcal{M}\left(
E_{1},...,E_{n};F\right)  ,$ $u_{j}\in\mathcal{L}\left(  G_{j},E_{j}\right)  $
for $j=1,...,n$ and $t\in\mathcal{L}\left(  F;H\right)  ,$ then $tA\left(
u_{1},...,u_{n}\right)  \in\mathcal{M}\left(  G_{1},...,G_{n};H\right)  .$
\end{definition}

\vspace{0,15cm} For each fixed $n$, $\mathcal{M}_{n}=\bigcup_{E_{1}%
,...,E_{n},F\text{ Banach}}\mathcal{M}\left(  E_{1},...,E_{n};F\right)  $ is
called $n$-linear multi-ideal.

\vspace{0,15cm}

\begin{definition}
A normed (or quasi-normed) ideal of multilinear mappings
\newline$\left(  \mathcal{M},\left\Vert \cdot\right\Vert _{\mathcal{M}%
}\right)  $ is an multilinear ideal provided with a function $\left\Vert
\cdot\right\Vert _{\mathcal{M}}:\mathcal{M}\longrightarrow\left[
0,\infty\right)  ,$ such that:

$(i)$ $\left\Vert \cdot\right\Vert _{\mathcal{M}}$ constrained to
$\mathcal{M}\left(  E_{1},...,E_{n};F\right)  $ is a norm (or quasi-norm),
with constant not depending on space, possibly depending only on $n$, for any
Banach spaces $E_{1},...,E_{n}$, $F$ and all $n\in\mathbb{N}$;

$(ii)$ $\left\Vert id_{\mathbb{K}^{n}}\right\Vert _{\mathcal{M}}=1$, where
$id_{\mathbb{K}^{n}}:\mathbb{K}^{n}\longrightarrow\mathbb{K}$ is given by
$id_{\mathbb{K}^{n}}\left(  x_{1},...,x_{n}\right)  =x_{1}\cdots x_{n}$ for
all $n\in\mathbb{N}$;

$(iii)$ If $M\in\mathcal{M}\left(  E_{1},...,E_{n};F\right)  $, $u_{j}%
\in\mathcal{L}\left(  G_{j},E_{j}\right)  $ for $j=1,...,n$ and $t\in
\mathcal{L}\left(  F;H\right)  ,$ then
\[
\left\Vert tM\left(  u_{1},...,u_{n}\right)  \right\Vert _{\mathcal{M}}
\leq\left\Vert t\right\Vert \left\Vert M\right\Vert _{\mathcal{M}}\left\Vert
u_{1}\right\Vert \cdots\left\Vert u_{n}\right\Vert .
\]
\end{definition}

\vspace{0,15cm} If $n$ is a fixed positive integer, under the above
conditions, we say that $\mathcal{M}_{n}$ is an normed (quasi-normed) ideal of
$n$-linear mappings. When the components $\mathcal{M}\left(  E_{1}%
,...,E_{n};F\right)  $ are complete with respect to $\left\Vert \cdot
\right\Vert _{\mathcal{M}}$, we say that $\mathcal{M}$ is a complete
multilinear ideal. The same is said about $\mathcal{M}_{n}.$

\vspace{0,15cm} Norms on ideals of multilinear mappings behave similar to the
case of ideals of linear mappings, so that $\left\Vert M\right\Vert
\leq\left\Vert M\right\Vert _{\mathcal{M}}$ for any $M$ in $\mathcal{M}$.

\vspace{0,15cm}

\begin{definition}
An ideal of homogeneous polynomials, or simply an ideal of polynomials is a
subclass $\mathcal{Q}$ of the class of all continuous homogeneous polynomials
between Banach spaces such that for all $n\in\mathbb{N}$ an any Banach spaces
$E$ and $F$, the components $\mathcal{Q}\left(  ^{n}E;F\right)  =\mathcal{P}%
\left(  ^{n}E;F\right)  \cap\mathcal{Q}$ satisfy:

$(i)$ $\mathcal{Q}\left(  ^{n}E;F\right)  $ is a vector subspace of
$\mathcal{P}\left(  ^{n}E;F\right)  $ containing the $n$-homogeneous finite
type polynomials;

$(ii)$ The ideal property: if $u\in\mathcal{L}\left(  G;E\right)  $,
$P\in\mathcal{Q}\left(  ^{n}E;F\right)  $ and $t\in\mathcal{L}\left(
F;H\right)  ,$ then $tPu\in\mathcal{Q}\left(  ^{n}G;H\right)  .$
\end{definition}

If $n\in\mathbb{N}$ is fixed, $\mathcal{Q}_{n}:=\bigcup_{E,F\text{ Banach}%
}\mathcal{Q}\left(  ^{n}E;F\right)  $ is called the ideal of $n$-homogeneous polynomials.

\vspace{0,15cm}

\begin{definition}
A normed (or quasi-normed) ideal of polynomials
$\left(  \mathcal{Q},\left\Vert \cdot\right\Vert _{\mathcal{Q}}\right)  $ is
an ideal of polynomials if there is a function $\left\Vert \cdot\right\Vert
_{\mathcal{Q}}:\mathcal{Q}\longrightarrow\left[  0,\infty\right)  ,$ such that:

$(i)$ $\left\Vert \cdot\right\Vert _{\mathcal{Q}}$ constrained to
$\mathcal{Q}\left(  ^{n}E;F\right)  $ is a norm (or quasi-norm) with constant
not depending on space, possibly depending only on $n$, for any Banach spaces
$E$ and $F$ and all $n\in\mathbb{N}$;

$(ii)$ $\left\Vert id_{\mathbb{K}}\right\Vert _{\mathcal{Q}}=1$, where
$id_{\mathbb{K}}:\mathbb{K}\longrightarrow\mathbb{K}$ is given by
$id_{\mathbb{K}}\left(  x\right)  =x^{n}$;

$(iii)$ If $u\in\mathcal{L}\left(  G,E\right)  ,$ $P\in\mathcal{Q}\left(
^{n}E;F\right)  $, and $t\in\mathcal{L}\left(  F;H\right)  ,$ then $\left\Vert
tPu\right\Vert _{\mathcal{Q}}\leq\left\Vert t\right\Vert \left\Vert
P\right\Vert _{\mathcal{Q}}\left\Vert u\right\Vert ^{n}.$
\end{definition}

If the components $\mathcal{Q}\left(  ^{n}E;F\right)  $ are complete with
respect to $\left\Vert \cdot\right\Vert _{Q},$ we say that $\mathcal{Q}$ is a
Banach ideal (or quasi-Banach ideal). Similarly one proceeds to $\mathcal{Q}%
_{n}.$

\vspace{0,15cm} The definition and results below show that it is always
possible to obtain a (complete) ideal of polynomials from the multi-ideals.

\vspace{0,15cm}

\begin{definition}
Let $\mathcal{M}$ be a quasi-normed ideal of multilinear mappings. The class
\[
\mathcal{P}_{\mathcal{M}}=\left\{  P\in\mathcal{P}^{n};\check{P}\in
\mathcal{M},n\in\mathbb{N}\right\}  \ ,
\]
with $\left\Vert P\right\Vert _{\mathcal{P}_{\mathcal{M}}}:=\left\Vert
\check{P}\right\Vert _{\mathcal{M}},$ is called ideal of polynomials generated
by the ideal $\mathcal{M}$.
\end{definition}

\vspace{0,15cm}
\begin{proposition}
Let $\mathcal{M}$ be a complete ideal of multilinear
mappings. Then $\mathcal{P}_{\mathcal{M}}$ is a Banach ideal of polynomials.
\end{proposition}

The following remark is useful for next results:
\vspace{0,15cm}
\begin{remark}
\label{des1} If $T \in\mathcal{L}(E;F)$ and $\varphi_{i}\in F^{^{\prime}%
},\ i=1,...,m$ , then we have
\begin{align*}
||(\varphi_{i} \circ T)_{i=1}^{m}||_{w,p^{*}}  &  = \sup_{y \in B_{E}}
||(\varphi_{i} (T(y)))_{i=1}^{m}||_{p^{*}}\\
&  = ||T|| \sup_{y \in B_{E}} \left\Vert \left(  \varphi_{i} \left(
\frac{T(y)}{||T||}\right)  \right)  _{i=1}^{m}\right\Vert _{p^{*}}\\
&  \leq||T|| \sup_{h \in B_{F}} ||(\varphi_{i} (h))_{i=1}^{m}||_{p^{*}}\\
&  = ||T||\, ||(\varphi_{i})_{i=1}^{m}||_{w,p^{*}}\ .
\end{align*}
\end{remark}

\vspace{0,15cm} Let us denote by $\mathcal{D}_{p}$ the class
of all linear operators between Banach spaces that are Cohen strongly
$p$-summing. We will show that $\left(  \mathcal{D}_{p}, d_{p}\right)  $ is a
complete normed ideal of linear operators.

\vspace{0,15cm}
\begin{proposition}
If $1 < p < \infty$, then $\left(  \mathcal{D}_p,
d_p \right)  $ is a complete normed ideal of linear operators.
\end{proposition}

\begin{proof}
We know that the components $\mathcal{D}_{p}(E;F)$ are normed spaces (with
$d_{p}(\cdot)$) for any Banach spaces $E$ and $F$. Straightforward calculations show that $d_p(id_\mathbb{K})=1$ and every component $\mathcal{D}_{p}(E;F)$ is complete with norm $d_p(\cdot)$ and contains the finite rank linear operators.

Ideal property: Let $A_{1} \in\mathcal{L}(E_{0};E)$, $T \in\mathcal{D}%
_{p}(E;F)$ and $A_{2} \in\mathcal{L}(F;F_{0})$. For all $m \in\mathbb{N}$,
using Remark \ref{des1}, we have
\begin{align}
\label{des2}\sum_{i=1}^{m} |\varphi_{i}((A_{2} \circ T \circ A_{1}) (x_{i}))|
&  = \sum_{i=1}^{m} |(\varphi_{i} \circ A_{2})(T(A_{1}(x_{i}))|\nonumber\\
&  \leq d_{p}(T) ||(A_{1}x_{i})_{i=1}^{m}||_{p}\,||(\varphi_{i} \circ
A_{2})_{i=1}^{m}||_{w,p^{*}}\nonumber\\
&  \leq d_{p}(T) ||A_{1}||\,||(x_{i})_{i=1}^{m}||_{p}\,||A_{2}||\,
||(\varphi_{i})_{i=1}^{m}||_{w,p^{*}}\nonumber\\
&  = ||A_{2}||\, d_{p}(T)\, ||A_{1}||\, ||(x_{i})_{i=1}^{m}||_{p}\,
||(\varphi_{i})_{i=1}^{m}||_{w,p^{*}}\\
&  = C\, ||(x_{i})_{i=1}^{m}||_{p}\,||(\varphi_{i})_{i=1}^{m}||_{w,p^{*}%
}\ \nonumber
\end{align}
and thus $A_{2} \circ T \circ A_{1} \in\mathcal{D}_{p}(E_{0};F_{0})$. Note
also that, for (\ref{des2}), we get
\[
d_{p}(A_{2} \circ T \circ A_{1}) \leq||A_{2}||\, d_{p}(T)\, ||A_{1}||\ .
\]
\end{proof}

\vspace{0,15cm}
Let us denote by $\mathcal{L}_{Coh,p}$ and by $\mathcal{P}_{Coh,p}$ the classes of all multilinear operators and polynomials between Banach spaces that are Cohen strongly $p$-summing. The proof of next Proposition is obtained in a similar way to the previous and will be omitted:

\vspace{0,15cm}
\begin{proposition}
If $1 < p < \infty$ and $n \in \mathbb{N}$, then $\left( \mathcal{L}_{Coh,p}^n , ||\cdot||_{Coh,p}\right)$ and $\left( \mathcal{P}_{Coh,p}^n , ||\cdot||_{Coh,p}\right)$ are a complete normed ideals of $n$-linear operators and $n$-homogeneous polynomials.
\end{proposition}

Finally, denoting by $\mathcal{L}_{mCoh,p}$ the class of all multilinear operators between Banach spaces that are multiple Cohen strongly $p$-summing, we show that $\left(\mathcal{L}_{mCoh,p}\ , ||\cdot||_{mCoh,p}\right)$ is a complete normed ideal of multilinear operators.

\vspace{0,15cm}
In that direction, Proposition \ref{inclusaom} shows that the components $\mathcal{L}_{mCoh,p}(E_1,...,E_n;F)$ contain the finite type multilinear operators, Proposition \ref{mCohnorm} ensures that these components are normed spaces and a standard calculus shows completeness. Since $||id_{\mathbb{K}^n}||=1$, the inequality $1 \leq ||id_{\mathbb{K}^n}||_{mCoh,p}$ is easily obtained. Moreover, as $\mathcal{L}_{Coh,p}$ is an ideal, $||id_{\mathbb{K}^n}||_{Coh,p}=1$ and Proposition \ref{inclusaom} gives us the inequality $||id_{\mathbb{K}^n}||_{mCoh,p} \leq 1$. It remains to show the ideal property.

\vspace{0,15cm}
\begin{theorem}
Let $n \in \mathbb{N}$. Then $\mathcal{L}_{mCoh,p}^n$ is a complete normed ideal of $n$-linear operators.
\end{theorem}

\begin{proof}
Let $A_i \in \mathcal{L}(H_i;E_i)$, $i=1,...,n$ , $T \in \mathcal{L}_{mCoh,p}(E_1,...,E_n;F)$ and $A \in \mathcal{L}(F;G)$. For all $m \in \mathbb{N}$, if $\varphi_{j_1,...,j_n} \in G^{'}$ and $x_j^{(i)} \in H_i$, $i=1,...,n, \ j_i=1,...,m, \ j=1,...,m$, Then
\begin{align*}
& \sum_{j_1,...,j_n=1}^m \left\vert \varphi_{j_1,...,j_n}\left(A \circ T \circ (A_1,...,A_n)\left(x_{j_1}^{(1)},...,x_{j_n}^{(n)}\right)\right) \right\vert \\
& = \sum_{j_1,...,j_n=1}^m \left\vert (\varphi_{j_1,...,j_n} \circ A)\left(T \left(A_1\left(x_{j_1}^{(1)}\right),...,A_n\left(x_{j_n}^{(n)}\right)\right)\right) \right\vert \\
& \leq ||T||_{mCoh,p} \left(\prod_{i=1}^n \left\Vert \left(A_i\left(x_{j}^{(i)}\right)\right)_{j=1}^m\right\Vert _p \right) ||(\varphi_{j_1,...,j_n} \circ A)_{j_1,...,j_n=1}^m||_{w,p^*} \\
& \leq ||A||\, ||T||_{mCoh,p}\, ||A_1|| \cdots ||A_n|| \left(\prod_{i=1}^n \left\Vert \left(x_{j}^{(i)}\right)_{j=1}^m\right\Vert _p \right) ||(\varphi_{j_1,...,j_n})_{j_1,...,j_n=1}^m||_{w,p^*}\ .
\end{align*}
So $A \circ T \circ (A_1,...,A_n) \in \mathcal{L}_{mCoh,p}(H_1,...,H_n;G)$ and
\[ ||A \circ T \circ (A_1,...,A_n)||_{mCoh,p} \leq ||A||\, ||T||_{mCoh,p}\, ||A_1|| \cdots ||A_n|| \ .\]
\end{proof}


\vspace{0,15cm}
\begin{definition}\label{defpolmult}
The class of multiple Cohen strongly $p$-summing $n$-homogeneous polynomials is the class
\[ \mathcal{P}_{mCoh,p}^n := \left\{P \in \mathcal{P}^n; \check{P} \in \mathcal{L}_{mCoh,p}^n\right\} \ .\]
\end{definition}
Furthermore, with the norm given by $||P||_{\mathcal{P}_{mCoh,p}} := ||\check{P}||_{mCoh,p}$,
we obtain a complete normed ideal of polynomials generated by the ideal $\mathcal{L}_{mCoh,p}$.

\vspace{0,15cm}
\section{Cohen strongly summing sequences: coherence and compatibility}\label{sc}

We show that the ideals of Cohen strongly $p$-summing polynomials/multilinear operators are coherent and compatible, as defined by Pellegrino and Ribeiro \cite{joilson11}. The same properties are shown to the class of multiple Cohen strongly $p$-summing operators.

First, we establish the concept of coherence and compatibility that we use.

Let $\left(  \mathcal{U}_{n},\mathcal{M}_{n}\right)_{n=1}^{N}$ be a sequence of pairs, where each $\mathcal{U}_{n}$ is a (quasi-)normed ideal of $n$-homogeneous polynomials and each $\mathcal{M}_{n}$ is a (quasi-)normed ideal of $n$-linear operators. The parameter $N$ can possibly be infinite.

\vspace{0,15cm}
\begin{definition}[Compatible pair of ideals]\label{defcomp}
Let $\mathcal{U}$ be a normed operator
ideal and $N\in\left(  \mathbb{N\smallsetminus}\left\{  1\right\}  \right)
\cup\left\{  \infty\right\}  $. A sequence $\left(  \mathcal{U}_{n}%
,\mathcal{M}_{n}\right)  _{n=1}^{N}$ with $\mathcal{U}_{1}=\mathcal{M}%
_{1}=\mathcal{U}$ is compatible with $\mathcal{U}$ if there exist positive
constants $\alpha_{1},\alpha_{2},\alpha_{3},\alpha_{4}$ such that for all
Banach spaces $E$ and $F,$ the following conditions hold for all $n\in\left\{
2,...,N\right\}  $:

\vspace{0,15cm}
\noindent (CP1) If $k\in\left\{  1,\ldots,n\right\}  ,$ $T\in\mathcal{M}_{n}\left(
E_{1},...,E_{n};F\right)  $ and $a_{j}\in E_{j}$ for all $j\in\left\{
1,...,n\right\}  \mathbb{\smallsetminus}\left\{  k\right\}  $, then
\[
T_{a_{1},\ldots,a_{k-1},a_{k+1},\ldots,a_{n}}\in\mathcal{U}\left(
E_{k};F\right)
\]
and
\[
\left\Vert T_{a_{1},\ldots,a_{k-1},a_{k+1},\ldots,a_{n}}\right\Vert
_{\mathcal{U}}\leq\alpha_{1}\left\Vert T\right\Vert _{\mathcal{M}_{n}%
}\left\Vert a_{1}\right\Vert \ldots\left\Vert a_{k-1}\right\Vert \left\Vert
a_{k+1}\right\Vert \ldots\left\Vert a_{n}\right\Vert .
\]
(CP2) If $P\in\mathcal{U}_{n}\left(  ^{n}E;F\right)  $ and $a\in E$, then
$P_{a^{n-1}}\in\mathcal{U}\left(  E;F\right)  $ and%
\[
\left\Vert P_{a^{n-1}}\right\Vert _{\mathcal{U}}\leq\alpha_{2}\left\Vert
\overset{\vee}{P}\right\Vert _{\mathcal{M}_{n}}\left\Vert a\right\Vert
^{n-1}.
\]
(CP3) If $u\in\mathcal{U}\left(  E_{n};F\right)  ,$ $\gamma_{j}\in
E_{j}^{\ast}$ for all $j=1,....,n-1$, then
\[
\gamma_{1}\cdots\gamma_{n-1}u\in\mathcal{M}_{n}\left(  E_{1},..,E_{n}%
;F\right)
\]
and
\[
\left\Vert \gamma_{1}\cdots\gamma_{n-1}u\right\Vert _{\mathcal{M}_{n}}%
\leq\alpha_{3}\left\Vert \gamma_{1}\right\Vert ...\left\Vert \gamma
_{n-1}\right\Vert \left\Vert u\right\Vert _{\mathcal{U}}.
\]
(CP4) If $u\in\mathcal{U}\left(  E;F\right)  ,$ $\gamma\in E^{\ast}$, then
$\gamma^{n-1}u\in\mathcal{U}_{n}\left(  ^{n}E;F\right)  $ and
\[
\left\Vert \gamma^{n-1}u\right\Vert _{\mathcal{U}_{n}}\leq\alpha_{4}\left\Vert
\gamma\right\Vert ^{n-1}\left\Vert u\right\Vert _{\mathcal{U}}.
\]
(CP5) $P$ belongs to $\mathcal{U}_{n}\left(  ^{n}E;F\right)  $ if, and only
if, $\overset{\vee}{P}$ belongs to $\mathcal{M}_{n}\left(  ^{n}E;F\right)  $.
\end{definition}

\vspace{0,15cm}
\begin{definition}
[Coherent pair of ideals]Let $\mathcal{U}$ be a normed operator
ideal and $N\in\mathbb{N}\cup\left\{  \infty\right\}  .$ A sequence $\left(
\mathcal{U}_{k},\mathcal{M}_{k}\right)  _{k=1}^{N},$ with $\mathcal{U}%
_{1}=\mathcal{M}_{1}=\mathcal{U}$, is coherent if there exist positive
constants $\beta_{1},\beta_{2},\beta_{3},\beta_{4}$ such that for all Banach
spaces $E$ and $F$ the following conditions hold for $k=1,...,N-1$:

\vspace{0,15cm}
\noindent (CH1) If $T\in\mathcal{M}_{k+1}\left(  E_{1},...,E_{k+1};F\right)  $ and
$a_{j}\in E_{j}$ for $j=1,\ldots,k+1,$ then
\[
T_{a_{j}}\in\mathcal{M}_{k}\left(  E_{1},\ldots,E_{j-1},E_{j+1},\ldots
,E_{k+1};F\right)
\]
and
\[
\left\Vert T_{a_{j}}\right\Vert _{\mathcal{M}_{k}}\leq\beta_{1}\left\Vert
T\right\Vert _{\mathcal{M}_{k+1}}\left\Vert a_{j}\right\Vert .
\]
(CH2) If $P\in\mathcal{U}_{k+1}\left(  ^{k+1}E;F\right)  ,$ $a\in E$, then
$P_{a}$ belongs to $\mathcal{U}_{k}\left(  ^{k}E;F\right)  $ and
\[
\left\Vert P_{a}\right\Vert _{\mathcal{U}_{k}}\leq\beta_{2}\left\Vert
\overset{\vee}{P}\right\Vert _{\mathcal{M}_{k+1}}\left\Vert a\right\Vert .
\]
(CH3) If $T\in\mathcal{M}_{k}\left(  E_{1},...,E_{k};F\right)  ,$ $\gamma\in
E_{k+1}^{\ast}$, then
\[
\gamma T\in\mathcal{M}_{k+1}\left(  E_{1},...,E_{k+1};F\right)  \text{ and
}\left\Vert \gamma T\right\Vert _{\mathcal{M}_{k+1}}\leq\beta_{3}\left\Vert
\gamma\right\Vert \left\Vert T\right\Vert _{\mathcal{M}_{k}}.
\]
(CH4)If $P\in\mathcal{U}_{k}\left(  ^{k}E;F\right)  ,$ $\gamma\in E^{\ast}$,
then
\[
\gamma P\in\mathcal{U}_{k+1}\left(  ^{k+1}E;F\right)  \text{ and }\left\Vert
\gamma P\right\Vert _{\mathcal{U}_{k+1}}\leq\beta_{4}\left\Vert \gamma
\right\Vert \left\Vert P\right\Vert _{\mathcal{U}_{k}}.
\]
(CH5) For all $k=1,...,N,$ $P$ belongs to $\mathcal{U}_{k}\left(
^{k}E;F\right)  $ if, and only if, $\overset{\vee}{P}$ belongs to
$\mathcal{M}_{k}\left(  ^{k}E;F\right)  $.
\end{definition}

\vspace{0,15cm}
According to the above definitions, coherence does not necessarily imply compatibility. However, under restrictions on the constants $\beta_i$, $i=1,...,4$, we have:

\vspace{0,15cm}
\begin{proposition}[Pellegrino and Ribeiro, \cite{joilson11}]\label{coerImpcomp}
If $\left(  \mathcal{U}_{n},\mathcal{M}_{n}\right)  _{n=1}^{N}$ is coherent with $\beta_{1}=\beta_{2}=\beta_{3}=\beta_{4}=1$, then is compatible with the ideal $\mathcal{U}_{1}=\mathcal{M}_{1}=\mathcal{U}$.
\end{proposition}

\vspace{0,15cm}
We now show that the sequence $(\mathcal{P}_{Coh,p}^n\ ,\mathcal{L}_{Coh,p}^n)_{n=1}^\infty$, composed by the ideal of $n$-homogeneous polynomials and ideals of Cohen strongly $p$-summing $n$-linear operators, is coherent and compatible with the ideal of Cohen strongly  $p$-summing linear operators.

\vspace{0,15cm}
\begin{theorem} The sequence $(\mathcal{P}_{Coh,p}^n,\mathcal{L}_{Coh,p}^n)_{n=1}^\infty$ is coherent and compatible with the ideal $\mathcal{D}_p$.
\end{theorem}

\begin{proof}
Using Proposition \ref{coerImpcomp}, we show that the sequence is coherent with constants $\beta_i=1$, $i=1,...,4$.

\vspace{0,15cm}
$(CH5)$ See Section \ref{ol}.

$(CH1)$ Since $T \in \mathcal{L}_{Coh,p}(E_1,...,E_{n+1};F)$, we have, $\forall m \in \mathbb{N}$,
\begin{align*}
\sum_{i=1}^m |\varphi_i(T_{a_1}(x_i^{(1)},...,x_i^{(n)}))|  &  = \sum_{i=1}^m |\varphi_i(T(a_1,x_i^{(1)},...,x_i^{(n)}))| \\
& \leq ||T||_{Coh,p} \left( \sum_{i=1}^m ||a_1||^p ||x_i^{(1)}||^p \cdots ||x_i^{(n)}||^p\right)^{1/p} ||(\varphi_i)_{i=1}^m||_{w,p^*} \\
& = ||T||_{Coh,p} ||a_1|| \left( \sum_{i=1}^m ||x_i^{(1)}||^p \cdots ||x_i^{(n)}||^p\right)^{1/p} ||(\varphi_i)_{i=1}^m||_{w,p^*} \ ,
\end{align*}
so $T_{a_1} \in \mathcal{L}_{Coh,p}(E_2,...,E_{n+1};F)$ and $||T_{a_1}||_{Coh,p} \leq ||T||_{Coh,p} ||a_1||$.
By proceeding in a similar way, we found that $T_{a_j} \in \mathcal{L}_{Coh,p}(E_1,...,E_{j-1},E_{j+1},...,E_{n+1};F)$,  $j=2,...,n$ and \[||T_{a_j}||_{Coh,p} \leq ||T||_{Coh,p} ||a_j||.\]

$(CH2)$ Since $P \in \mathcal{P}_{Coh,p}(^{n+1}E;F)$, we have, $\forall m \in \mathbb{N}$,
\begin{align*}
\sum_{i=1}^m |\varphi_i(P_a(x_i))|  &  = \sum_{i=1}^m |\varphi_i(\check{P}(a,x_i,...,x_i))| \\
& \leq ||\check{P}||_{Coh,p} \left( \sum_{i=1}^m ||a||^p ||x_i||^{np} \right)^{1/p} ||(\varphi_i)_{i=1}^m||_{w,p^*} \\
& = ||\check{P}||_{Coh,p} ||a|| \left( \sum_{i=1}^m ||x_i||^{np}\right)^{1/p} ||(\varphi_i)_{i=1}^m||_{w,p^*}
\end{align*}
and therefore $P_a \in \mathcal{P}_{Coh,p}(^nE;F)$ and $||P_a||_{Coh,p} \leq ||\check{P}||_{Coh,p} ||a||$.

$(CH3)$ If $T \in \mathcal{L}_{Coh,p}(E_1,...,E_n;F)$ and $\gamma \in E_{n+1}^{'}$, we obtain, $\forall m \in \mathbb{N}$,
\begin{align*}
& \sum_{i=1}^m |\varphi_i(\gamma T(x_i^{(1)},...,x_i^{(n)},x_i^{(n+1)}))| \\
&  = \sum_{i=1}^m |\varphi_i(T(x_i^{(1)},...,x_i^{(n)})\gamma(x_i^{(n+1)}))| \\
& = \sum_{i=1}^m |\varphi_i(T(x_i^{(1)},...,x_i^{(n)} \gamma(x_i^{(n+1)})))| \\
& \leq ||T||_{Coh,p} \left(\sum_{i=1}^m ||x_i^{(1)}||^p \cdots ||x_i^{(n-1)}||^p \, ||x_i^{(n)}\gamma (x_i^{(n+1)})||^p \right)^{1/np}\ ||(\varphi_i)_{i=1}^m||_{w,p^*} \\
&  \leq ||T||_{Coh,p} ||\gamma || \left(\sum_{i=1}^m \prod_{j=1}^{n+1} ||x_i^{(j)}||^p \right)^{1/p}\ ||(\varphi_i)_{i=1}^m||_{w,p^*}\
\end{align*}
and so $\gamma T$ belongs to $\mathcal{L}_{Coh,p}(E_1,...,E_n,E_{n+1};F)$ and  $||\gamma T||_{Coh,p} \leq ||T||_{Coh,p} ||\gamma||$.

$(CH4)$ If $P \in \mathcal{P}_{Coh,p}(^nE;F)$ and $\gamma \in E^{'}$, we have, $\forall m \in \mathbb{N}$,
\begin{align*}
\sum_{i=1}^m |\varphi_i(\gamma P(x_i))| &  = \sum_{i=1}^m |\varphi_i(\gamma (x_i) P(x_i))| \\
& = \sum_{i=1}^m |\varphi_i(P((\gamma (x_i))^{1/n}x_i))| \\
& \leq ||P||_{Coh,p} \left(\sum_{i=1}^m ||(\gamma (x_i))^{1/n}x_i||^{np} \right)^{1/p}\ ||(\varphi_i)_{i=1}^m||_{w,p^*} \\
&  \leq ||P||_{Coh,p} ||\gamma || \left(\sum_{i=1}^m ||x_i||^{np} \right)^{1/p}\ ||(\varphi_i)_{i=1}^m||_{w,p^*}\
\end{align*}
which implies  that $\gamma P$ belongs to $\mathcal{P}_{Coh,p}(^{n+1}E;F)$ and  $||\gamma P||_{Coh,p} \leq ||P||_{Coh,p} ||\gamma||$.

Thus, by Proposition \ref{coerImpcomp}, $(\mathcal{P}_{Coh,p}^n,\mathcal{L}_{Coh,p}^n)_{n=1}^\infty$ is coherent and compatible with the ideal $\mathcal{D}_p$.
\end{proof}

Now, we show that the ideals of polynomials multiple Cohen strongly $p$-summing operators form sequences which are coherent and compatible with the ideal of Cohen strongly  $p$-summing linear operators.

\vspace{0,15cm}
\begin{theorem} The sequence $(\mathcal{P}_{mCoh,p}^n,\mathcal{L}_{mCoh,p}^n)_{n=1}^\infty$ is coherent and compatible with the ideal $\mathcal{D}_p$ of Cohen strongly $p$-summing linear operators.
\end{theorem}

\begin{proof}
Again, using Proposition \ref{coerImpcomp}, we show that the sequence is coherent with constants $\beta_i=1$, $i=1,...,4$.

\vspace{0,15cm}
$(CH5)$ Follows immediately from Definition \ref{defpolmult}.

\vspace{0,15cm}
$(CH1)$ We show that if $T \in \mathcal{L}_{mCoh,p}(E_1,...,E_{n+1};F)$, then \\ $T_{a_{n+1}} \in \mathcal{L}_{mCoh,p}(E_1,...,E_{n};F)$.
For all $m \in \mathbb{N}$, taking the functionals

\begin{displaymath}
\varphi_{j_1,...,j_{n},j_{n+1}} = \left\{
\begin{array}{l}
\varphi_{j_1,...,j_{n}},\ \mathrm{if} \ j_{n+1} = 1 \\
0, \ \mathrm{if}\ j_{n+1}=2,...,m\ ,
\end{array} \right.
\end{displaymath}
with $j_1,...,j_{n},j_{n+1}=1,...,m$ , since $T \in \mathcal{L}_{mCoh,p}(E_1,...,E_{n+1};F)$, we obtain,

\begin{align*}
& \sum_{j_1,...,j_{n}=1}^m |\varphi_{j_1,...,j_{n}}(T_{a_{n+1}}(x_{j_1}^{(1)},...,x_{j_n}^{(n)}))|  = \sum_{j_1,...,j_{n},j_{n+1}=1}^m |\varphi_{j_1,...,j_{n},j_{n+1}}(T(x_{j_1}^{(1)},...,x_{j_n}^{(n)},a_{n+1}))| \\
& \leq ||T||_{mCoh,p}\, ||a_{n+1}||\, \left\Vert\left(x_{j}^{(1)}\right)_{j=1}^m\right\Vert_p \cdots \left\Vert \left(x_{j}^{(n)}\right)_{j=1}^m\right\Vert_p\, ||(\varphi_{j_1,...,j_{n},j_{n+1}})_{j_1,...,j_{n},j_{n+1}=1}^m||_{w,p^*} \\
& = ||T||_{mCoh,p}\, ||a_{n+1}||\, \left\Vert\left(x_{j}^{(1)}\right)_{j=1}^m\right\Vert_p \cdots \left\Vert \left(x_{j}^{(n)}\right)_{j=1}^m\right\Vert_p\, ||(\varphi_{j_1,...,j_{n}})_{j_1,...,j_{n}=1}^m||_{w,p^*}\ ,
\end{align*}
so $T_{a_{n+1}} \in \mathcal{L}_{mCoh,p}(E_1,...,E_{n};F)$ and $||T_{a_{n+1}}||_{mCoh,p} \leq ||T||_{mCoh,p}\, ||a_{n+1}||$.
Analogously,  \[T_{a_j} \in \mathcal{L}_{mCoh,p}(E_1,...,E_{j-1},E_{j+1},...,E_{n+1};F),\  j=1,...,n\] and $||T_{a_j}||_{mCoh,p} \leq ||T||_{mCoh,p}\, ||a_j||$.

\vspace{0,15cm}
$(CH2)$ Having in mind that $(P_a)^\vee = \check{P}_a$, the result follows as a consequence of the two previous items.

$(CH3)$ Let $\gamma \in E_{n+1}^{'}$. For all positive integer $m$ we have
\begin{align}\label{par1} \nonumber
& \sum_{j_1,...,j_{n+1}=1}^m \left\vert \varphi_{j_1,...,j_{n+1}}\left(\gamma T(x_{j_1}^{(1)},...,x_{j_n}^{(n)},x_{j_{n+1}}^{(n+1)})\right)\right\vert \nonumber \\
&  = \sum_{j_1,...,j_{n+1}=1}^m \left\vert \varphi_{j_1,...,j_{n+1}}\left(T(x_{j_1}^{(1)},...,x_{j_n}^{(n)}\gamma(x_{j_{n+1}}^{(n+1)}))\right)\right\vert \ ,
\end{align}
and the expression (\ref{par1}) can be rewritten as
\begin{equation}\label{par2}
\sum_{j_n=1}^{m^2} \sum_{j_1,...,j_{n-1}=1}^m \left\vert \tilde{\varphi}_{j_1,...,j_{n}}\left(T(z_{j_1}^{(1)},...,z_{j_n}^{(n)})\right)\right\vert \\
\end{equation}
with the choices
\vspace{0,2cm}
\begin{displaymath}
\left\{
\begin{array}{l}
z_{j_i}^{(i)}= x_{j_i}^{(i)},\ j_i=1,...,m\ \mathrm{and}\ \ i=1,...,n-1\\
z_{j_n}^{(n)}= x_{j_n}^{(n)}\gamma (x_{1}^{(n+1)}),\ j_n=1,...,m \\
z_{m+j_n}^{(n)}= x_{j_n}^{(n)}\gamma (x_{2}^{(n+1)}),\ j_n=1,...,m \\
\ \ \ \vdots \\
z_{(m-1)m +j_n}^{(n)}= x_{j_n}^{(n)}\gamma (x_{m}^{(n+1)}),\ j_n=1,...,m \\
\end{array} \right.
\end{displaymath}
\vspace{0,2cm}
and
\vspace{0,2cm}
\begin{displaymath}
\left\{
\begin{array}{l}
\tilde{\varphi}_{j_1,...,j_{n}} = \varphi_{j_1,...,j_{n},1},\  j_n=1,...,m\\
\tilde{\varphi}_{j_1,...,m+j_{n}} = \varphi_{j_1,...,j_{n},2},\  j_n=1,...,m\\
\ \ \ \vdots \\
\tilde{\varphi}_{j_1,...,(m-1)m+j_{n}} = \varphi_{j_1,...,j_{n},m},\  j_n=1,...,m \ .\\
\end{array} \right.
\end{displaymath}
In fact,

\begin{align*}
&  \sum_{j_1,...,j_{n+1}=1}^m \left\vert \varphi_{j_1,...,j_{n+1}}\left(T(x_{j_1}^{(1)},...,x_{j_n}^{(n)}\gamma(x_{j_{n+1}}^{(n+1)}))\right)\right\vert \\
& = \sum_{j_{n+1}=1}^m \left( \sum_{j_1,...,j_{n}=1}^m \left\vert \varphi_{j_1,...,j_{n+1}}\left(T(x_{j_1}^{(1)},...,x_{j_n}^{(n)}\gamma(x_{j_{n+1}}^{(n+1)}))\right)\right\vert \right)\\
& = \sum_{j_1,...,j_{n}=1}^m \left\vert \varphi_{j_1,...,j_{n},1}\left(T(x_{j_1}^{(1)},...,x_{j_n}^{(n)}\gamma(x_{1}^{(n+1)}))\right)\right\vert +\\
& + \sum_{j_1,...,j_{n}=1}^m \left\vert \varphi_{j_1,...,j_{n},2}\left(T(x_{j_1}^{(1)},...,x_{j_n}^{(n)}\gamma(x_{2}^{(n+1)}))\right)\right\vert + \cdots\\
& \cdots + \sum_{j_1,...,j_{n}=1}^m \left\vert \varphi_{j_1,...,j_{n},m}\left(T(x_{j_1}^{(1)},...,x_{j_n}^{(n)}\gamma(x_{m}^{(n+1)}))\right)\right\vert \\
& = \sum_{j_1,...,j_{n}=1}^m \left\vert \tilde{\varphi}_{j_1,...,j_{n}}\left(T(z_{j_1}^{(1)},...,z_{j_n}^{(n)})\right)\right\vert +\\
& + \sum_{j_1,...,j_{n}=1}^m \left\vert \tilde{\varphi}_{j_1,...,m+j_{n}}\left(T(z_{j_1}^{(1)},...,z_{m+j_n}^{(n)})\right)\right\vert + \cdots\\
& \cdots + \sum_{j_1,...,j_{n}=1}^m \left\vert \tilde{\varphi}_{j_1,...,(m-1)m+j_{n}}\left(T(z_{j_1}^{(1)},...,z_{(m-1)m+j_n}^{(n)})\right)\right\vert \\
& = \sum_{j_n=1}^{m^2} \sum_{j_1,...,j_{n-1}=1}^m \left\vert \tilde{\varphi}_{j_1,...,j_{n}}\left(T(z_{j_1}^{(1)},...,z_{j_n}^{(n)})\right)\right\vert \ .
\end{align*}
In this way, if $T \in \mathcal{L}_{mCoh,p}(E_1,...,E_{n};F)$, using (\ref{par1}) and (\ref{par2}), we get
\begin{align*}
& \sum_{j_1,...,j_{n+1}=1}^m \left\vert \varphi_{j_1,...,j_{n+1}}\left(\gamma T(x_{j_1}^{(1)},...,x_{j_n}^{(n)},x_{j_{n+1}}^{(n+1)})\right)\right\vert \\
& = \sum_{j_n=1}^{m^2} \left(\sum_{j_1,...,j_{n-1}=1}^m \left\vert \tilde{\varphi}_{j_1,...,j_{n}}\left(T(z_{j_1}^{(1)},...,z_{j_n}^{(n)})\right)\right\vert \right)\\
& =  \sum_{j_1,...,j_{n-1},j_n=1}^{m,...,m,m^2} \left\vert \tilde{\varphi}_{j_1,...,j_{n}}\left(T(z_{j_1}^{(1)},...,z_{j_n}^{(n)})\right)\right\vert \\
& \leq ||T||_{mCoh,p}\, \left\Vert \left(z_{j_n}^{(n)}\right)_{j_n=1}^{m^2}\right\Vert_p \prod_{i=1}^{n-1} \left\Vert \left(z_{j_i}^{(i)}\right)_{j_i=1}^m\right\Vert_p \,  ||(\tilde{\varphi}_{j_1,...,j_{n}})_{j_1,...,j_{n-1},j_n=1}^{m,...,m,m^2}||_{w,p^*} \\
& = ||T||_{mCoh,p}\, \left\Vert \left(x_{j_n}^{(n)}\gamma (x_{j_{n+1}}^{(n+1)})\right)_{j_n,j_{n+1}=1}^{m}\right\Vert_p \prod_{i=1}^{n-1} \left\Vert \left(x_{j}^{(i)}\right)_{j=1}^m\right\Vert_p \,  ||(\varphi_{j_1,...,j_{n+1}})_{j_1,...,j_{n+1}=1}^{m}||_{w,p^*} \\
& \overset{(*)}{\leq} ||T||_{mCoh,p}\,||\gamma|| \prod_{i=1}^{n+1} \left\Vert \left(x_{j}^{(i)}\right)_{j=1}^m\right\Vert_p \,  ||(\varphi_{j_1,...,j_{n+1}})_{j_1,...,j_{n+1}=1}^{m}||_{w,p^*} \ ,
\end{align*}
where in the transition $(*)$, we are using the fact that
\[ \left\Vert \left(x_{j}\gamma (y_{k})\right)_{j,k=1}^{m}\right\Vert_p  =\left( \sum_{j,k=1}^m |\gamma (y_{k})|^{p}\, ||x_{j}||^{p}\right)^{1/p}
  =\left\Vert \left(\gamma (y_{k})\right)_{k=1}^m\right\Vert_p \, \left\Vert \left(x_{j}\right)_{j=1}^m\right\Vert_p \]

and the continuity of $\gamma$.

Therefore, $\gamma T\in\mathcal{L}_{mCoh,p}\left(  E_{1},\ldots,E_{n+1};F\right)  $ and
\[ \left\Vert \gamma T\right\Vert _{mCoh,p}\leq \left\Vert T\right\Vert _{mCoh,p} \left\Vert\gamma\right\Vert .\]

$(CH4)$ Let $\gamma \in E^{'}$ and $S_n$ the set of all permutations of the set $\{1,...,n\}$. Note that we can build $(\gamma P)^\vee$ as
\begin{align*}
(\gamma P)^\vee(x_1,...,x_{n+1}) & = \frac{1}{(n+1)!} \sum_{\sigma \in S_{n+1}} \gamma(x_{\sigma(k)}) \check{P}(x_{\sigma(1)},\overset{[k]}{\cdots}, x_{\sigma(n+1)}) \\
& = \frac{1}{(n+1)!} \left[ \gamma(x_{\sigma(1)}) \sum_{\sigma \in S_{n}} \check{P}(x_{\sigma(2)},\cdots , x_{\sigma(n+1)}) + \cdots \right.\\
& \left.\cdots + \gamma(x_{\sigma(n+1)})\sum_{\sigma \in S_{n}} \check{P}(x_{\sigma(1)},\cdots , x_{\sigma(n)}) \right]\\
& = \frac{1}{(n+1)!} \left[\gamma(x_{1}) n! \check{P}(x_{2},\cdots , x_{n+1}) + \cdots + \gamma(x_{n+1}) n! \check{P}(x_{1},\cdots , x_{n}) \right] \\
& = \frac{1}{n+1} \sum_{k=1}^{n+1} \gamma(x_{k})  \check{P}(x_{1},\overset{[k]}{\cdots} , x_{n+1}) \ ,
\end{align*}
where $\overset{[k]}{\cdots}$ indicates that the $k$-th coordinate is not involved.

Using this fact, if $m \in \mathbb{N}$, we have

\begin{align*}
& \sum_{j_1,...,j_{n+1}=1}^m \left\vert \varphi_{j_1,...,j_{n+1}}\left((\gamma P)^\vee (x_{j_1}^{(1)},...,x_{j_{n+1}}^{(n+1)})\right)\right\vert  \\
& = \frac{1}{n+1} \sum_{j_1,...,j_{n+1}=1}^m \left\vert \varphi_{j_1,...,j_{n+1}}\left(\sum_{k=1}^{n+1} \gamma(x_{j_k}^{(k)})  \check{P}(x_{j_1}^{(1)},\overset{[k]}{\cdots} , x_{j_{n+1}}^{(n+1)})\right)\right\vert  \\
& \leq \frac{1}{n+1} \left( \sum_{j_1,...,j_{n+1}=1}^m \left(\sum_{k=1}^{n+1} \left\vert \varphi_{j_1,...,j_{n+1}}\left( \gamma(x_{j_k}^{(k)})  \check{P}(x_{j_1}^{(1)},\overset{[k]}{\cdots} , x_{j_{n+1}}^{(n+1)})\right)\right\vert \right)\right)  \\
& = \frac{1}{n+1} \left\Vert \left(\sum_{k=1}^{n+1} \left\vert \varphi_{j_1,...,j_{n+1}}\left( \gamma(x_{j_k}^{(k)})  \check{P}(x_{j_1}^{(1)},\overset{[k]}{\cdots} , x_{j_{n+1}}^{(n+1)})\right)\right\vert \right)_{j_1,...,j_{n+1}=1}^m \right\Vert_1  \\
\end{align*}
\begin{align*}
& \leq \frac{1}{n+1} \sum_{k=1}^{n+1} \left\Vert \left(\left\vert \varphi_{j_1,...,j_{n+1}}\left( \gamma(x_{j_k}^{(k)})  \check{P}(x_{j_1}^{(1)},\overset{[k]}{\cdots} , x_{j_{n+1}}^{(n+1)})\right)\right\vert \right)_{j_1,...,j_{n+1}=1}^m \right\Vert_1  \\
& = \frac{1}{n+1} \sum_{k=1}^{n+1} \left( \sum_{j_1,...,j_{n+1}=1}^m \left\vert \varphi_{j_1,...,j_{n+1}}\left( \gamma(x_{j_k}^{(k)})  \check{P}(x_{j_1}^{(1)},\overset{[k]}{\cdots} , x_{j_{n+1}}^{(n+1)})\right)\right\vert\right) \\
& = \frac{1}{n+1} \left[ \left( \sum_{j_1,...,j_{n+1}=1}^m \left\vert \varphi_{j_1,...,j_{n+1}}\left( \check{P}(\gamma(x_{j_1}^{(1)})x_{j_2}^{(2)},\cdots , x_{j_{n+1}}^{(n+1)})\right)\right\vert\right) + \cdots \right. \\
& \left. \cdots +  \left( \sum_{j_1,...,j_{n+1}=1}^m \left\vert \varphi_{j_1,...,j_{n+1}}\left( \check{P}(\gamma(x_{j_{n+1}}^{(n+1)})x_{j_1}^{(1)},\cdots , x_{j_{n}}^{(n)})\right)\right\vert\right) \right] \ .
\end{align*}

Hence
\begin{align}\label{parcelas}
& \sum_{j_1,...,j_{n+1}=1}^m \left\vert \varphi_{j_1,...,j_{n+1}}\left((\gamma P)^\vee (x_{j_1}^{(1)},...,x_{j_{n+1}}^{(n+1)})\right)\right\vert   \\
& = \frac{1}{n+1} \left[ \left( \sum_{j_1,...,j_{n+1}=1}^m \left\vert \varphi_{j_1,...,j_{n+1}}\left( \check{P}(\gamma(x_{j_1}^{(1)})x_{j_2}^{(2)},\cdots , x_{j_{n+1}}^{(n+1)})\right)\right\vert\right) + \cdots \right. \nonumber \\
& \left. \cdots +  \left( \sum_{j_1,...,j_{n+1}=1}^m \left\vert \varphi_{j_1,...,j_{n+1}}\left( \check{P}(\gamma(x_{j_{n+1}}^{(n+1)})x_{j_1}^{(1)},\cdots , x_{j_{n}}^{(n)})\right)\right\vert\right) \right] \nonumber \ .
\end{align}

Then, by the same argument used to demonstrate the property $(CH3)$, each part of the (\ref{parcelas}), for example the first, can be written as
\begin{equation*}
\sum_{j_2=1}^{m^2} \sum_{j_3,...,j_{n+1}=1}^m \left\vert \tilde{\varphi}_{j_2,...,j_{n+1}}\left(\check{P}(z_{j_2}^{(2)},...,z_{j_{n+1}}^{(n+1)})\right)\right\vert \\
\end{equation*}
for convenient choices of $\tilde{\varphi}_{j_2,...,j_{n+1}}$ and $z_{j_k}^{(k)}$, $k=2,...,n+1$, and therefore, as shown in property $(CH3)$, we have
\begin{align*}
& \sum_{j_1,...,j_{n+1}=1}^m \left\vert \varphi_{j_1,...,j_{n+1}}\left( \check{P}(\gamma(x_{j_1}^{(1)})x_{j_2}^{(2)},\cdots , x_{j_{n+1}}^{(n+1)})\right)\right\vert \\
& \leq ||\check{P}||_{mCoh,p}\,||\gamma|| \prod_{i=1}^{n+1} \left\Vert \left(x_{j}^{(i)}\right)_{j=1}^m\right\Vert_p \,  ||(\varphi_{j_1,...,j_{n+1}})_{j_1,...,j_{n+1}=1}^{m}||_{w,p^*}\ .
\end{align*}

So, returning to (\ref{parcelas}), we finally obtain
\begin{align*}
& \sum_{j_1,...,j_{n+1}=1}^m \left\vert \varphi_{j_1,...,j_{n+1}}\left((\gamma P)^\vee (x_{j_1}^{(1)},...,x_{j_{n+1}}^{(n+1)})\right)\right\vert   \\
& = \frac{1}{n+1} \left[ \left( \sum_{j_1,...,j_{n+1}=1}^m \left\vert \varphi_{j_1,...,j_{n+1}}\left( \check{P}(\gamma(x_{j_1}^{(1)})x_{j_2}^{(2)},\cdots , x_{j_{n+1}}^{(n+1)})\right)\right\vert\right) + \cdots \right. \\
& \left. \cdots +  \left( \sum_{j_1,...,j_{n+1}=1}^m \left\vert \varphi_{j_1,...,j_{n+1}}\left( \check{P}(\gamma(x_{j_{n+1}}^{(n+1)})x_{j_1}^{(1)},\cdots , x_{j_{n}}^{(n)})\right)\right\vert\right) \right]  \\
& \leq \frac{1}{n+1} \left[ ||\check{P}||_{mCoh,p}\,||\gamma|| \prod_{i=1}^{n+1} \left\Vert \left(x_{j}^{(i)}\right)_{j=1}^m\right\Vert_p \,  ||(\varphi_{j_1,...,j_{n+1}})_{j_1,...,j_{n+1}=1}^{m}||_{w,p^*} + \cdots \right. \\
& \left. \cdots +  ||\check{P}||_{mCoh,p}\,||\gamma|| \prod_{i=1}^{n+1} \left\Vert \left(x_{j}^{(i)}\right)_{j=1}^m\right\Vert_p \,  ||(\varphi_{j_1,...,j_{n+1}})_{j_1,...,j_{n+1}=1}^{m}||_{w,p^*} \right] \\
& = ||\check{P}||_{mCoh,p}\,||\gamma|| \prod_{i=1}^{n+1} \left\Vert \left(x_{j}^{(i)}\right)_{j=1}^m\right\Vert_p \,  ||(\varphi_{j_1,...,j_{n+1}})_{j_1,...,j_{n+1}=1}^{m}||_{w,p^*} \ .
\end{align*}

Thus $\gamma P \in \mathcal{P}_{mCoh,p}(^{n+1}E;F)$ and
\begin{equation*}
||\gamma P||_{mCoh,p} \leq ||\check{P}||_{mCoh,p}\,||\gamma|| = ||P||_{mCoh,p}\,||\gamma|| \ .
\end{equation*}

Therefore, by Proposition \ref{coerImpcomp}, $(\mathcal{P}_{mCoh,p}^n,\mathcal{L}_{mCoh,p}^n)_{n=1}^\infty$ is coherent and compatible with the ideal $\mathcal{D}_p$ of Cohen strongly $p$-summing linear operators.
\end{proof}

\vspace{0,15cm}
\section{Holomorphy types and the ideal of multiple Cohen strongly multilinear operators}\label{ht}

In the previous section we shown that the notion of multiple Cohen strongly $p$-summing multilinear operators is well behaved from the viewpoint of the theory of ideals of operators. In this section we show that this class is also well behaved from the viewpoint of holomorphy.

The following definition is essentially the same as that given by Nachbin \cite{nachbin69}:

\vspace{0,15cm}
\begin{definition}[Botelho et. al., \cite{studia}]
A global holomorphy type is a class $\mathcal{P}_H$ of continuous homogeneous polynomials between Banach spaces such that for all natural $n$ and Banach spaces $E$ and $F$ the components $\mathcal{P}_H(^nE;F) := \mathcal{P}(^nE;F) \cup \mathcal{P}_H$ satisfy:

\vspace{0,15cm}
$(i)$ $\mathcal{P}_H(^nE;F)$ is a linear Banach subspace of $\mathcal{P}(^nE;F)$ endowed with a norm denoted by $P \mapsto ||P||_H$;

$(ii)$ $\mathcal{P}_H(^0E;F) = F$ is a linear normed space for all $E$ and $F$;

$(iii)$ there is a constant $\sigma \geq 1$ such that for all $n \in \mathbb{N}$, $k\leq n$, $a \in E$ and any Banach spaces $E$ and $F$, with $P \in \mathcal{P}_H(^nE;F)$,
\begin{equation*}
\hat{d}^kP(a) \in \mathcal{P}_H(^kE;F) \ \mathrm{and}\ \left\Vert \frac{1}{k!}\hat{d}^kP(a) \right\Vert_H \leq \sigma^n ||P||_H ||a||^{n-k}\ ,
\end{equation*}
where $\hat{d}^kP(a)$ is the $k$th differential of $P$ at $a$.
\end{definition}

If we have quasi-norms instead of norms (each $\mathcal{P}_H(^nE; F)$ is a complete
quasi-normed space with quasi-norm constants not depending on the under-
lying spaces E and F, but possibly depending on n), we say that $\mathcal{P}_H$ is a
global quasi-holomorphy type.

For the next definition, we use the notation $(^nE,G;F)$ instead of $(E,\overset{(n)}{\cdots},E,G;F)$.

\vspace{0,15cm}
\begin{definition}[Botelho et. al., \cite{studia}]
Let $\mathcal{J}$ be a class of continuous multilinear mappings between Banach spaces such that for all $n \in \mathbb{N}$ and Banach spaces $E_1,...,E_n$
and $F$, the component $\mathcal{J}(E_1,...,E_n;F) := \mathcal{L}(E_1,...,E_n;F) \cup \mathcal{J}$ is
a linear subspace of $\mathcal{L}(E_1,...,E_n; F)$ equipped with a norm denoted by $||\cdot||_\mathcal{J}$ . We say that $\mathcal{J}$ has property (B) if there is $C \geq 1$
such that for every $n \in \mathbb{N}$, any Banach spaces E and F and every
$A \in \mathcal{J}(^nE,\mathbb{K};F)$ symmetric in the first $n$ variables, occurs
\begin{equation*}
A1 \in \mathcal{J}(^nE;F)\ \ \mathrm{and}\ \  ||A1||_\mathcal{J} \leq C||A||_\mathcal{J} \ ,
\end{equation*}
where $A1\ :\, E^n \rightarrow F$ is defined by $A1(x_1,...,x_n) := A(x_1,...,x_n,1)$.
\end{definition}

\vspace{0,15cm}
\begin{theorem}[Botelho et. al., \cite{studia}]
If the Banach ideal $\mathcal{M}$ of multilinear operators has property (B) with constant C, then the Banach ideal $\mathcal{P}_\mathcal{M}$ of polynomials generated by $\mathcal{M}$ is a global holomorphy type with constant $\sigma = 2C$.
\end{theorem}

\vspace{0,15cm}
We show that the ideal $\mathcal{L}_{mCoh,p}$ of multiple Cohen strongly $p$-summing multilinear operators has the property (B) and therefore the class $\mathcal{P}_{mCoh,p}$ of multiple Cohen strongly $p$-summing polynomials is a global holomorphy type.

\vspace{0,15cm}
\begin{theorem} The complete ideal $\mathcal{L}_{mCoh,p}$ has property (B) with constant $C =1$. Therefore, the ideal $\mathcal{P}_{mCoh,p}$ of multiple Cohen strongly $p$-summing polynomials is a global holomorphy type with constant $\sigma =2$.
\end{theorem}

\begin{proof}
Let be $n \in \mathbb{N}$, $E$ and $F$ Banach spaces and $T \in \mathcal{L}_{mCoh,p}(^nE,\mathbb{K};F)$. For all positive integer $m$, we define
\begin{equation}\label{esc}
\varphi_{j_1,...,j_{n},j_{n+1}} = \left\{
\begin{array}{l}
\varphi_{j_1,...,j_{n}},\ \mathrm{if} \ j_{n+1} = 1 \\
0, \ \mathrm{if}\ j_{n+1}=2,...,m\ ,
\end{array} \right.
\ \ \mathrm{and}\ \
y_{j_{n+1}} = \left\{
\begin{array}{l}
1,\ \mathrm{if} \ j_{n+1} = 1 \\
0, \ \mathrm{if}\ j_{n+1}=2,...,m\ ,
\end{array} \right.
\end{equation}
with $j_1,...,j_{n},j_{n+1}=1,...,m$.  Thus, we have
\begin{align*}
& \sum_{j_1,...,j_{n}=1}^m \left\vert \varphi_{j_1,...,j_{n}}\left(T1 (x_{j_1}^{(1)},...,x_{j_n}^{(n)})\right)\right\vert\\
& = \sum_{j_1,...,j_{n}=1}^m \left\vert \varphi_{j_1,...,j_{n}}\left(T(x_{j_1}^{(1)},...,x_{j_n}^{(n)},1)\right)\right\vert\\
& \overset{(\ref{esc})}{=} \sum_{j_1,...,j_{n+1}=1}^m \left\vert \varphi_{j_1,...,j_{n+1}}\left(T(x_{j_1}^{(1)},...,x_{j_n}^{(n)}, y_{j_{n+1}})\right)\right\vert\\
& \leq ||T||_{mCoh,p}\, \left\Vert \left(y_{j}\right)_{j=1}^m\right\Vert_p \prod_{i=1}^{n} \left\Vert \left(x_{j}^{(i)}\right)_{j=1}^m\right\Vert_p \,  ||(\varphi_{j_1,...,j_{n+1}})_{j_1,...,j_{n+1}=1}^{m}||_{w,p^*} \\
& = ||T||_{mCoh,p}\, \prod_{i=1}^{n} \left\Vert \left(x_{j}^{(i)}\right)_{j=1}^m\right\Vert_p \,  ||(\varphi_{j_1,...,j_{n}})_{j_1,...,j_{n}=1}^{m}||_{w,p^*}\ ,
\end{align*}
from which $T1 \in \mathcal{L}_{mCoh,p}(^nE;F)$ e $||T1||_{mCoh,p} \leq ||T||_{mCoh,p}$. Therefore, $\mathcal{L}_{mCoh,p}$ has the property (B) with constant $C=1$.
\end{proof}


\end{document}